\documentclass[12pt, reqno]{amsart}

\usepackage{amsmath,amssymb,latexsym,mathtools, comment}
\usepackage[dvipsnames]{xcolor}
\usepackage{graphicx,xcolor,mathrsfs}
\usepackage{cite}
\usepackage[utf8]{inputenc}
\usepackage[english]{babel}
\usepackage[font={scriptsize}]{caption}
\usepackage{empheq}
\usepackage[most]{tcolorbox}
\usepackage{color}
\usepackage{bbm}

\tcbset{highlight math style={enhanced,
  colframe=blue,colback=blue!1!white,arc=0pt,boxrule=0.5pt}}

\usepackage[margin=3cm, a4paper]{geometry}

\DeclareMathOperator{\supp}{supp}

\DeclareMathOperator*{\esslim}{ess\,lim}

\newtheorem{theorem}{Theorem}[section]
\newtheorem{lemma}[theorem]{Lemma}
\newtheorem{proposition}[theorem]{Proposition}
\newtheorem{remark}[theorem]{Remark}
\newtheorem{definition}[theorem]{Definition}
\newtheorem{corollary}[theorem]{Corollary}
\newtheorem*{main-theorem}{Main Theorem}
\newtheorem*{remark*}{Remark}

\numberwithin{equation}{section}

\usepackage{enumitem}
\usepackage{mathrsfs}

\usepackage[colorlinks=true]{hyperref}
\hypersetup{urlcolor=blue, citecolor=red, linkcolor=blue}

\usepackage{scalerel,stackengine}
\stackMath
\newcommand\widecheck[1]{%
\savestack{\tmpbox}{\stretchto{%
  \scaleto{%
    \scalerel*[\widthof{\ensuremath{#1}}]{\kern-.6pt\bigwedge\kern-.6pt}%
    {\rule[-\textheight/2]{1ex}{\textheight}}
  }{\textheight}%
}{0.5ex}}%
\stackon[1pt]{#1}{\scalebox{-1}{\tmpbox}}%
}

\newlist{abbrv}{itemize}{1}
\setlist[abbrv,1]{label=,labelwidth=1in,align=parleft,itemsep=0.1\baselineskip,leftmargin=!}

\let\olddefinition\definition
\renewcommand{\definition}{\olddefinition\normalfont}
\let\oldremark\remark
\renewcommand{\remark}{\oldremark\normalfont}
\newcommand{\sgn}{\text{sgn}}

\newcommand{\R}{\mathbb{R}}

\newcommand{\N}{\mathbb{N}}

\newcommand{\F}{\mathcal{F}}

\renewcommand{\k}{\kappa}
\newcommand{\x}{\xi}

\newcommand{\e}{\varepsilon}

\newcommand{\dd}{\mathrm{d}}

\newcommand{\loc}{\mathrm{loc}}
\newcommand{\norm}[2]{\|#1\|_{{#2}}}
\newcommand{\snorm}[2]{|#1|_{{#2}}}

\renewcommand{\b}{\beta}
\renewcommand{\d}{\delta}
\newcommand{\cop}[1]{\langle#1\rangle}

\DeclareMathOperator*{\esssup}{ess\,sup}
\DeclareMathOperator*{\essinf}{ess\,inf}
\DeclareMathOperator*{\esslimsup}{ess\,lim\,sup}
\DeclareMathOperator*{\essliminf}{ess\,lim\,inf}

\title[One-sided Hölder smoothing of negative order dispersive equations]{One-sided Hölder regularity of global weak solutions of negative order dispersive equations}

\author[Maehlen]{Ola I.H. M{\ae}hlen}
\address{Department of Mathematics, University of Oslo, 0851 Oslo, Norway}
\email{olamaeh@uio.no}

\author[Xue]{Jun Xue}
\address{Department of Mathematical Sciences, Norwegian University of Science and Technology, 7491 Trondheim, Norway}
\email{jun.xue@ntnu.no}


\allowdisplaybreaks

\begin{document}

\maketitle

\begin{abstract}
We prove global existence, uniqueness and stability of entropy solutions with $L^2$ initial data for a general family of negative order dispersive equations. These weak solutions are found to satisfy one-sided Hölder conditions whose coefficients decay in time. The latter result controls the height of solutions and further provides a way to bound the maximal lifespan of classical solutions from their initial data.
\end{abstract}

\section{Introduction}
We consider the initial value problem
\begin{align}\label{eq: theMainEquation}
 \left\{ \begin{aligned}
 &u_t+\tfrac{1}{2}(u^2)_x =(G\ast u)_x, \qquad (t,x)\in \mathbb R^+ \times \mathbb R,\\
&\hspace{24pt}u(0,x)=u_0(x), \hspace{56pt} x\in \mathbb R,\\
   \end{aligned}
  \right.
\end{align}
for initial data $u_0\in L^2(\R)$ and an even convolution kernel $G\in L^1(\R)$ admitting an integrable weak derivative $G'\eqqcolon K \in L^1(\R)$. Included in this family of equations is the Burgers--Poisson equation
\begin{align}\label{eq: burgersPoissonEquation}
    u_t + \tfrac{1}{2}(u^2)_x = \bigg(\int_{-\infty}^\infty -\tfrac{1}{2}e^{-|x-y|}u(t,y)\dd y\bigg)_x
\end{align}
 which in \cite{MR0483954} is derived as a model for shallow water waves.
 
 \subsection{Outline of main results}
 
 The paper can be divided in two parts: Section \ref{sec: wellPosedness} establishes the well-posedness of entropy solutions of \eqref{eq: theMainEquation}, while Section \ref{sec: longTermBounds} demonstrates the one-sided Hölder regularity that the solutions enjoy. To the best of our knowledge, these results are new. It was shown in \cite{MR3527628} that the Burgers--Poisson equation \eqref{eq: burgersPoissonEquation} admits unique entropy solutions with $L^1$ data that satisfy one sided Lipschitz conditions. Still, the results here add new insight also for Burgers--Poisson: The $L^2$ setting is more natural (albeit harder) to work in due to the dispersive right-hand side. For the $L^2$ norm of a solution is guaranteed to be non-increasing in time, which can be used to deduce a one-sided \textit{smoothing} effect of \eqref{eq: theMainEquation}. In particular, our Corollary \ref{cor: lipschitzSmoothingOfBurgersPoisson} shows that the one-sided Lipschitz coefficients of solutions of Burgers--Poisson can at worst behave like `$\mathrm{const} + 1/t$' whereas the corresponding expression in \cite{MR3527628} takes the form $O(te^{t}+1/t)$. An interesting consequence of having an explicit smoothing effect of \eqref{eq: theMainEquation}, as given by Theorem \ref{thm: oneSidedHolderRegularity}, is that it provides a necessary condition on terminal data when seeking to solve the backward problem; this is exploited in the proof of Corollary \ref{cor: maximalLifespanForClassicalSolutions} which bounds the lifespan of classical solutions of \ref{eq: theMainEquation}.

We give a brief discussion of our results which are presented in Section \ref{sec: mainResults}.
The first main result, Theorem \ref{thm: wellPosedness}, provides existence, uniqueness and $L^2$ stability for entropy solutions of \eqref{eq: theMainEquation} -- as defined by Def. \ref{def: entropySolution} -- for initial data in $L^2\cap L^\infty(\R)$. Corollary \ref{cor: extensionToPureL2Data} then extends this result in a unique and continuous manner to pure $L^2$ data. The results are proved in Section \ref{sec: wellPosedness}. There, uniqueness and stability is proved through a variation of Kru\v{z}kov's doubling of variables technique \cite{1970SbMat..10..217K}, while existence follows from an operator splitting argument. While there are less laborious approaches for proving existence (fixed point methods, vanishing viscosity), operator splitting has the advantage of allowing for a straight forward analysis of the regularizing effect of \eqref{eq: theMainEquation} which constitutes the second part of our results.

The second main result, Theorem \ref{thm: oneSidedHolderRegularity}, guarantees one-sided Hölder regularity for entropy solutions of \eqref{eq: theMainEquation}, and it is proved in Section \ref{sec: longTermBounds}. Like the classical Ole\v{ı}nik estimate \eqref{eq: classicalOleinikEstimateForBurgersSolutions} for Burgers' equation, this one-sided regularity \textit{improves} over time. The proof is based on an operator splitting approach, used to study the evolution of the quantity $\omega(t,h)\coloneqq \sup_{x}(u(t,x+h)-u(t,x))$, for $t,h>0$ and a solution $u$. As seen by Lemma \ref{lem: generalEvolutionOfOleinikEstimateUnderBurgersSolutionMap}, the nonlinearity in \eqref{eq: theMainEquation} has a smoothing effect on $\omega$. The dispersive term on the other hand has no clear convenient effect on $\omega$, and it is instead treated as a source term that we limit using the non-increasing $L^2$ norm of $u$ (as done when combining Lemma \ref{lem: generalEvolutionOfOleinikEstimateUnderConvolutionSolutionMap} and \ref{lem: theBoundOnSupFromL2Norm}). 

The result has two interesting consequences. First, Corollary \ref{cor: decayingHeightBoundForEntropySolutions} provides an explicit height bound for a solution $u$ in terms of $\norm{K}{L^1(\R)}$, $\norm{u_0}{L^2(\R)}$ and the time $t$. This bound decays initially like $1/t^{\frac{1}{3}}$ and converges to a positive constant for large times. Generally, the height of a solution will not tend to zero due to the existence of solitary waves \cite{MR2979975} for several instances of \eqref{eq: theMainEquation}. Second, Corollary \ref{cor: maximalLifespanForClassicalSolutions} bounds the lifespan of classical solutions of \eqref{eq: theMainEquation} provided the initial data satisfies a skewness condition \eqref{eq: skewnessCondition}. One may wonder how these classical solutions break down, and wave breaking is the natural candidate. A proof of this is beyond the scope of this paper, but not hard to obtain; demonstrating that \eqref{eq: theMainEquation} is classically well posed for times $t\lesssim -1/\inf_x u_0'(x)$ (the hyperbolic lifespan) would leave wave breaking as the only type of blow up (as we already have height bounds). We point out that our skewness condition \eqref{eq: skewnessCondition} differs from that of both \cite{MR3527628} and \cite{MR1668586}; neither imply the other.
 
 \subsection{Other dispersive equations}
 Central questions in the study of water wave model equations include well-posedness, persistence and non-persistence of solutions, the latter two exemplified by solitary and breaking waves. The answers depend intricately on the type of nonlinearity and dispersive term featured in the equation. In the case of a quadratic nonlinearity, the fractional Korteweg--de Vries  equation (fKdV)
\begin{align}\label{eq: defFractionalKDV}
    u_t + \tfrac{1}{2}(u^2)_x = (|D|^{\b}u)_x
\end{align}
where $\F(|D|^\b u)=|\xi|^\b \hat{u}$ and $\b\in\R$, has been suggested \cite{MR3188389} as a scale for studying how the strength of the dispersion affects the questions of well-posedness and water-wave features. To connect \eqref{eq: theMainEquation} to the fKdV setting, observe that our assumption on $G$ implies that $\widehat{G}(\xi)=o(|\xi|^{-1})$ as $|\x|\to\infty$ and so one may place \eqref{eq: theMainEquation} in the region $\b<-1$ for fKdV. However, $\widehat{G}$ will in our case be bounded, whereas $|\xi|^{\b}$ blows up at zero, and thus \eqref{eq: theMainEquation} can not match the low-frequency effect of negative order fKdV which assigns (very) high velocities to (very) low frequencies. This qualitative difference disappears in a periodic setting; the dispersion of fKdV on the torus is for $\b<-1$ precisely of the form assumed in \eqref{eq: theMainEquation}. We point out that the methods in this paper can be carried out on the torus; our results can thus be extended to periodic solutions of fKdV for $\b<-1$. With the relation between \eqref{eq: theMainEquation} and \eqref{eq: defFractionalKDV} accounted for, we now summarize a few results for the latter to sketch what one may expect of well-posedness and water-wave features in our case.

The fractional KdV equation of order $\b\in (\frac{6}{7},2]$ is globally well-posed in appropriate function spaces. The regions $\beta\in(\frac{6}{7},1)$ and $\b\in(1,2)$ are treated in \cite{MR3906854} and \cite{MR2754070} respectively, and there are numerous works on the well posedness for $\b=1$ (Benjamin--Ono equation) and $\b=2$ (KdV equation); see for example \cite{MR2291918} and \cite{MR3990604} and the references therein. For values $\b\leq \frac{6}{7}$ only local well-posedness results have been established \cite{MR3906854,MR3995034}. Still, numerical investigation \cite{MR3317254} suggests that fKdV is globally well-posed for dispersion as weak as $\b>\frac{1}{2}$, but not for $\b\leq \frac{1}{2}$; this is also conjectured in \cite{MR3188389}. One might expect the culprit of this loss of global well-posedness for weak dispersion, to be the appearance of breaking waves (shock formation), i.e. bounded solutions that develop infinite slope in finite time. In the negative order regime $\b<0$ this might be true: the occurrence of breaking waves have been proved for the case $\b=-2$ (Ostrovsky--Hunter equation) by \cite{MR2684307}, for the case $\b=-1$ (Burgers--Hilbert) by \cite{yang2020shock} and for the region $\b\in(-1,-\frac{1}{3})$ by \cite{MR3682673}. However, no such results exist in the positive order regime $\b>0$, and it is believed that instead other blowup phenomena occur in the range $\b\in(0,\frac{1}{2}]$ inhibiting global well-posedness; see the discussion in \cite{MR3317254, MR3188389} or \cite{MR3694646} where an example of $L^\infty$ blowup in finite time is constructed for the modified Benjamin--Ono equation. In the absence of classical global solutions, several authors have for the $\b<0$ regime turned to the concept of entropy solutions. Adapted from the study of hyperbolic conservation laws, entropy solutions are weak solutions that satisfy extra conditions -- the entropy inequalities -- automatically satisfied by classical solutions (whenever they exist). This solution concept allows for continuation past wave breaking and a global well-posedness theory may then be achieved. In \cite{MR3273174} existence and uniqueness of global entropy solutions for the Ostrovsky--Hunter equation ($\b=-2$) is established for appropriate initial data. Similarly, \cite{MR3248030} provides global entropy solutions for the Burgers--Hilbert equation $(\b=-1)$ and a partial uniqueness result. And as mentioned above, the Burgers--Poisson equation \eqref{eq: burgersPoissonEquation} is in \cite{MR3527628} shown to admit unique global entropy solutions for $L^1$ initial data. There, the authors also provide sufficient conditions on the initial data leading to wave breaking. This equation is not an isolated instance of \eqref{eq: theMainEquation} featuring wave breaking; \cite{MR1668586} shows that the phenomenon is present whenever $G\in C\cap L^1(\R)$ is symmetric and monotone on $\R^+$. More generally, our Corollary \ref{cor: maximalLifespanForClassicalSolutions} hints that every instance of \eqref{eq: theMainEquation} features wave breaking as explained above.

\subsection{The entropy formulation}\label{sec: entropyFormulation}
 We define the concept of entropy solutions on the function class $L^\infty_{\text{loc}}([0,\infty), L^\infty(\R))$, which here denotes the subspace of $L^\infty_\loc([0,\infty)\times\R)$ of functions $u(t,x)$ that are essentially bounded on $[0,T]\times\R$ for each $T>0$. We will in Section \ref{sec: mainResults} be more liberal in our definition of entropy solutions (allowing then for $L^2$ initial data) as explained after Corollary \ref{cor: extensionToPureL2Data}.
 
Necessary is the notion of an entropy pair $(\eta,q)$ for \eqref{eq: theMainEquation}, which is to say that
 \begin{align*}
     \eta\colon\R\to\R \text{ is smooth and convex, while } q'(u)=\eta'(u)u.
 \end{align*}
\begin{definition}\label{def: entropySolution}
For bounded initial data $u_0\in L^\infty(\R)$, we say that a function $u\in L^\infty_{\text{loc}}([0,\infty), L^\infty(\R))$ is an entropy solution of \eqref{eq: theMainEquation} if:
\begin{enumerate}
    \item  it satisfies for all non-negative $\varphi\in C^\infty_c(\R^+\times\R)$ and all entropy pairs $(\eta,q)$ of \eqref{eq: theMainEquation} the entropy inequality
\begin{equation}\label{eq: entropyInequality}
    \begin{split}
        &\int_0^\infty\int_{\R} \eta(u)\varphi_t+q(u)\varphi_x +\eta'(u) (K\ast u)\varphi \,\mathrm dx\mathrm dt\geq 0,
    \end{split}
\end{equation}
\item it assumes the initial data in $L^1_{\text{loc}}$ sense, that is
\begin{align*}
    \esslim_{t\searrow0}\int_{-r}^r |u(t,x)-u_0(x)| \mathrm{d}x=0,
\end{align*}
for all $r>0$.
\end{enumerate}
\end{definition}
The concept of entropy solutions lies between that of strong and weak solutions.
If $u\in L^\infty_{\text{loc}}([0,\infty), L^\infty(\R))\cap C^1(\R^+\times\R)$ is a classical solution of \eqref{eq: theMainEquation} then it is necessarily an entropy solution as multiplying \eqref{eq: theMainEquation} with $\eta'(u)\varphi$ and integrating by parts yields \eqref{eq: entropyInequality} as an equality. And if $u$ is an entropy solution of \eqref{eq: theMainEquation} then it is necessarily a weak solution as follows from considering the two entropy pairs $(\eta(u),q(u))= (u,\frac{1}{2}u^2)$ and $(\eta(u),q(u))=(-u,-\frac{1}{2}u^2)$ respectively.

\subsection{A fractional variation}
The exponents of the one-sided Hölder conditions provided by Theorem \ref{thm: oneSidedHolderRegularity} depend on the regularity of $K=G'$; the smoother $K$ is, the higher the exponent. More precisely, we attain the Hölder exponent $\frac{1+s}{2}$ if $|K|_{TV^s}<\infty$ where the latter seminorm is for $s\in [0,1]$ defined by 
\begin{align}\label{eq: definitionOfFractionalVariance}
    |K|_{TV^s} = \sup_{h>0}\frac{\norm{K(\cdot+h)-K}{L^1(\R)}}{h^s}.
\end{align}
When $s=1$ this seminorm coincides with the classical total variation of $K$, while $s=0$ gives twice the $L^1$ norm of $K$, and thus we necessarily have $|K|_{TV^0}<\infty$ as we assume $K\in L^1(\R)$. For $s\in(0,1)$ the seminorm is a measure of intermediate regularity between $L^1(\R)$ and $BV(\R)$. This seminorm does not coincide with the scaling invariant fractional variation from \cite{Love1937} used in \cite{MR3312048} to attain maximal smoothing effects for one-dimensional scalar conservation laws.

\section{Main results}\label{sec: mainResults}
We here present the two main results, Theorem \ref{thm: wellPosedness} and Theorem \ref{thm: oneSidedHolderRegularity} and corresponding corollaries. For a general discussion of the content given here, see the end of the above introduction. We start with Theorem \ref{thm: wellPosedness}, which provides a global well-posedness theory for entropy solutions of \eqref{eq: theMainEquation} with initial data in $L^2\cap L^\infty(\R)$. The theorem is established in Section \ref{sec: wellPosedness}.

\begin{theorem}\label{thm: wellPosedness}
For every initial data $u_0\in L^2\cap L^\infty(\R)$ there exists a unique entropy solution $u$ of \eqref{eq: theMainEquation}. The mapping $t\mapsto u(t)$ is continuous from $[0,\infty)$ to $L^2(\R)$ and $u(t)$ satisfies for all $t\geq0$ the bounds
\begin{align}\label{eq: boundsForWellPosednessThm}
     \norm{u(t)}{L^2(\R)}\leq& \norm{u_0}{L^2(\R)}, & \norm{u(t)}{L^\infty(\R)}\leq& e^{t\k}\norm{u_0}{L^\infty(\R)},
\end{align}
where $\k\coloneqq \norm{K}{L^1(\R)}$. Moreover, we have the following stability result: if two sequences $(t_k)_{k\in\N}\subset [0,\infty)$ and $(u_{0,k})_{k\in\N}\subset L^2\cap L^\infty(\R)$ admit the limits
\begin{align*}
    \lim_{k\to\infty}t_k=t, \quad \text{ and }\quad \lim_{k\to\infty}u_{0,k}=u_0 \quad \text{in }L^2(\R),
\end{align*}
where $u_0\in L^2\cap L^\infty(\R)$, then the corresponding entropy solutions satisfy
\begin{align*}
   \lim_{k\to\infty} u_k(t_k)= u(t) \quad \text{in }L^2(\R).
\end{align*}
\end{theorem}
 The following corollary is an extension of the result to a pure $L^2$ setting which is proved at the end of Subsection \ref{sec: L2ContinuityAndStability}.
\begin{corollary}[Global $L^2$ well-posedness]\label{cor: extensionToPureL2Data}
Equation \eqref{eq: theMainEquation} is globally well-posed for $L^2(\R)$ initial data in the following sense: The solution map $S\colon (t,u_0)\mapsto u(t)$ mapping $L^2\cap L^\infty(\R)$ initial data to the corresponding entropy solution evaluated at time $t\geq0$, extends uniquely to a jointly continuous map $S\colon[0,\infty)\times L^2(\R)\to L^2(\R)$. In particular, the $L^2$-bound, -continuity and -stability of Theorem \ref{thm: wellPosedness} carries over to all weak solutions provided by $S$. Moreover, for any $u_0\in L^2(\R)$, the corresponding weak solution $u(t,x)\coloneqq S(t,u_0)(x)$ is locally bounded in $(0,\infty)\times \R$ and satisfies the entropy inequalities \eqref{eq: entropyInequality}.
\end{corollary}

For the remainder of the section, we broaden the definition of an entropy solution: for $u_0\in L^2(\R)$ we say that $u$ is the corresponding entropy solution of \eqref{eq: theMainEquation} if, and only if, $u(t)=S(t,u_0)$, where $S$ is as in the previous corollary.

The second theorem infers one-sided Hölder regularity for the entropy solutions. The Hölder exponent depends on the regularity of $K=G'$, here measured using the fractional variation $|K|_{TV^s}$ defined in \eqref{eq: definitionOfFractionalVariance}. The theorem is proved in Section \ref{sec: longTermBounds}.

\begin{theorem}[One-sided Hölder regularity]\label{thm: oneSidedHolderRegularity} 
For initial data $u_0\in L^2(\R)$, let $u$ be the corresponding entropy solution of \eqref{eq: theMainEquation}, and let $s\in[0,1]$ be such that $|K|_{TV^s}<\infty$. Then $u$ satisfies the one-sided Hölder condition
\begin{align}\label{eq: oneSidedHolderCondition}
    u(t,x)-u(t,y)\leq a(t)(x-y)^{\frac{1+s}{2}},
\end{align}
for all $x\geq y$ and $t>0$, where the Hölder coefficient $a(t)$ is given by 
\begin{equation}\label{eq: holderCoefficient}
    a(t)=C_1(s) |K|_{TV^s}^{\frac{2+s}{3+2s}}\norm{u_0}{L^2(\R)}^{\frac{1+s}{3+2s}} + C_2(s)\frac{\norm{u_0}{L^2(\R)}^{\frac{1-s}{3}}}{t^{\frac{2+s}{3}}},
\end{equation}
for two constants $C_1(s)$ and $C_2(s)$ written out in \eqref{eq: C1AndC2}.
\end{theorem}
 Since $u(t)$ is in $L^2(\R)$ it is not necessarily true that $x\mapsto u(t,x)$ is well defined pointwise; in the previous theorem we have identified $u(t)$ with its left-continuous representation which exists due to Lemma \ref{lem: modulusOfGrowthImpliesLeftAndRightLimits}. 
 
 Since $K\in L^1(\R)$ and $|K|_{TV^0}=2\norm{K}{L^1(\R)}$, the $s=0$ case of Theorem \ref{thm: oneSidedHolderRegularity} is valid for any instance of \eqref{eq: theMainEquation}. In particular, entropy solutions of \eqref{eq: theMainEquation} are guaranteed to admit one-sided Hölder regularity of order $\frac{1}{2}$. In the case of the Burgers--Poisson equation, where $K =\tfrac{1}{2}\sgn(x)e^{-|x|}$ we find $|K|_{TV^1}=|K|_{TV}=2$ and so by the $s=1$ case of Theorem \ref{thm: oneSidedHolderRegularity} we get the following corollary.
 \begin{corollary}[One-sided Lipschitz smoothing of Burgers--Poisson]\label{cor: lipschitzSmoothingOfBurgersPoisson}
 For initial data $u_0\in L^2(\R)$, let $u$ be the corresponding entropy solution of the Burgers--Poisson equation \eqref{eq: burgersPoissonEquation}. Then $u$ satisfies the one-sided Lipschitz condition
 \begin{align*}
      u(t,x)-u(t,y)\leq&\,\bigg[12^{\frac{1}{5}}\norm{u_0}{L^2(\R)}^{\frac{2}{5}} + \frac{1}{t}\bigg](x-y),
 \end{align*}
 for all $x\geq y$ and $t>0$.
 \end{corollary}
While it was already established in \cite{MR3527628} that entropy solutions of the Burgers--Poisson equation are one-sided Lipschitz continuous, our result has the advantage of a Lipschitz coefficient that decreases with time. 

 We conclude this section with two less obvious corollaries of Theorem \ref{thm: oneSidedHolderRegularity}: a decaying height bound for entropy solutions of \eqref{eq: theMainEquation} and a maximal lifespan estimate for classical solutions.

\begin{corollary}[Height bound]\label{cor: decayingHeightBoundForEntropySolutions}
For initial data $u_0\in L^2(\R)$, let $u$ be the corresponding entropy solution of \eqref{eq: theMainEquation}. Then for all $t>0$ we have the height bound
\begin{align}\label{eq: heightBoundWhenSIsZero}
    \norm{u(t)}{L^\infty(\R)}\leq \Bigg[2^{\frac{11}{12}}3^{\frac{1}{3}}\norm{K}{L^1(\R)}^{\frac{1}{3}} + \frac{2^{\frac{5}{4}}}{t^{\frac{1}{3}}}\Bigg] \norm{u_0}{L^2(\R)}^{\frac{2}{3}}.
\end{align}
\end{corollary}
\begin{proof}\renewcommand{\qedsymbol}{}
See Appendix \ref{sec: proofsOfCorollaries}.
\end{proof}
Observe that together, the two height bounds \eqref{eq: boundsForWellPosednessThm} and \eqref{eq: heightBoundWhenSIsZero} imply that when $u_0\in L^2\cap L^\infty(\R)$ the corresponding entropy solution of \eqref{eq: theMainEquation} is globally bounded. 

For the final result of the section, we need to introduce the following seminorm
\begin{align}
    [u_0]_{s} \coloneqq \esssup_{\substack{x\in \R\\ h>0}}\bigg[\frac{u_0(x-h)-u_0(x)}{h^{\frac{1+s}{2}}}\bigg],
\end{align}
which is a (left) one-sided Hölder seminorm of exponent $\frac{1+s}{2}$.
\begin{corollary}[Maximal lifespan]\label{cor: maximalLifespanForClassicalSolutions}
There are universal constants $C,c>0$ such that: 
if initial data $u_0\in L^2\cap L^\infty(\R)$ satisfies the skewness condition
\begin{align}\label{eq: skewnessCondition}
    [u_0]_s^{3+2s}> c |K|_{TV^s}^{2+s}\norm{u_0}{L^2(\R)}^{1+s},
\end{align}
for some $s\in[0,1]$ such that $|K|_{TV^s}<\infty$,
then the lifespan $T$ of a classical solution $u\in L^\infty\cap C^1((0,T)\times \R)$ of \eqref{eq: theMainEquation} admitting $u_0$ as initial data must satisfy
\begin{align}\label{eq: asymptoticLifespandBound}
    T < C\Bigg[\frac{\norm{u_0}{L^2(\R)}^{1-s}}{[u_0]_s^{3}}\Bigg]^{\frac{1}{2+s}}.
\end{align}
\end{corollary}
\begin{proof}\renewcommand{\qedsymbol}{}
See Appendix \ref{sec: proofsOfCorollaries}.
\end{proof}

\section{Well posedness of entropy solutions}\label{sec: wellPosedness}
In this section, we provide for \eqref{eq: theMainEquation} a global well-posedness theory of entropy solutions as defined by Def. \ref{def: entropySolution}. In particular, the content of Theorem \ref{thm: wellPosedness} follows from Proposition \ref{prop: uniquenessAndL1Contraction}, Corollary \ref{cor: generalSolutionsAreLimitsOfBVSolutions} and Proposition \ref{prop: stabilityOfEntropySolutionsWithL2Data}; see the summary at the beginning of Subsection \ref{sec: L2ContinuityAndStability}. Corollary \ref{cor: extensionToPureL2Data} is also proved here at the end of Subsection \ref{sec: L2ContinuityAndStability}. For entropy solutions of \eqref{eq: theMainEquation}, the proofs of existence and uniqueness is the same for $L^2\cap L^\infty$ data as for $L^\infty$ data; only the $L^1$ setting allows for `shortcuts'. Thus for generality, many results in the two coming subsections will be presented for initial data $u_0\in L^\infty(\R)$. We also note that in these two subsections only Lemma \ref{lem: theApproximateSolutionMapConservesTheL2Norm} exploits the dispersive nature of \eqref{eq: theMainEquation}, that is, that $K=G'$ is odd.

\subsection{Uniqueness of entropy solutions}\label{sec: uniquenessOfEntropySolutions}  It is natural to start with the proof of uniqueness, as this equips us with a weighted $L^1$-contraction that can further be used in the existence proof. The involved weight $w_M^r(t,x)$ can be interpreted as a bound on the propagation of information for solutions of \eqref{eq: theMainEquation}. Its technical role in the coming proof is to serve as a subsolution of a dual equation, namely the one obtained from setting the square bracket in \eqref{eq: katoInequality} to zero. A similar method can be found in \cite{MR2305729} where nonlocal conservation laws are treated. 

The weight is constructed as follows. 
Writing $|K|$ to denote the function $x\mapsto |K(x)|$, we introduce for a parameter $t\geq0$ the operator $e^{t|K|\ast}$ mapping $L^p(\R)$ to itself for any $p\in[1,\infty]$, defined by
\begin{align}\label{eq: definitionOfExponentialConvolutionOperator}
    \Big(e^{t|K|\ast}f\Big)(x)= f(x) + \sum_{n=1}^\infty \Big((|K|\ast)^{ n}f\Big)(x)\frac{t^n}{n!},
\end{align}
where $(|K|\ast)^{n}$ represents the operation of convolving with $|K|$ repeatedly $n$ times. Observe that by repeated use of Young's convolution inequality we have for any $p\in[1,\infty]$ and $f\in L^p(\R)$
\begin{align}\label{eq: boundnessOfExponentialConvolutionOperator}
    \norm{e^{t|K|\ast}f}{L^p(\R)}\leq e^{t\k}\norm{f}{L^p(\R)},
\end{align}
where $\k\coloneqq \norm{K}{L^1(\R)}$.
For parameters $r,M\geq 0$, we further introduce
\begin{align}\label{eq: definitionOfTheCharacteristicFunction}
    \chi_M^r(t,x) = \begin{cases}
    1, & |x|<r+Mt,\\
    0, & \text{else},
    \end{cases}
\end{align}
and set
\begin{align}\label{eq: definitionOfTheWeight}
    w_M^r(t,x)=\Big(e^{t|K|\ast}\chi_M^r(t,\cdot)\Big)(x).
\end{align}
By \eqref{eq: boundnessOfExponentialConvolutionOperator}, this weight satisfies for $p\in[1,\infty]$ the bound
\begin{align}\label{eq: boundOnTheWeightForTheL1Contraction}
    \norm{w_M^r(t,\cdot)}{L^p(\R)}\leq e^{t\k} (2r+2Mt)^{\frac{1}{p}},
\end{align}
where the case $p=\infty$ is evaluated in a limit sense. Thus, $w_M^r(t,\cdot)\in L^1\cap L^\infty(\R)$ for all $t,r,M\geq 0$. With $w_M^r$ defined, we are ready to state Proposition \ref{prop: uniquenessAndL1Contraction} establishing the uniqueness of entropy solutions.
Although the following result is stated to hold for a.e. $t\geq 0$, it can be extended to all $t\geq 0$, as we later prove that entropy solutions of \eqref{eq: theMainEquation} are continuous when viewed as $L^1_{\loc}(\R)$-valued time-dependent functions. 

\begin{proposition}\label{prop: uniquenessAndL1Contraction}
Let $u,v\in L^\infty_{\mathrm{loc}}([0,\infty),L^\infty(\R))$ be entropy solutions of \eqref{eq: theMainEquation} with $u_0,v_0\in L^\infty(\R)$ as initial data. Then, for any $r>0$ and a.e. $t\geq 0$ we have the weighted $L^1$-contraction
\begin{align}\label{eq: fullL1Contraction}
\int_{-r}^{r}|u(t,x)-v(t,x)|dx\leq \int_{-\infty}^\infty |u_0(x)-v_0(x)|w_M^r(t,x)dx,
\end{align}
where $w_M^r$ is given by \eqref{eq: definitionOfTheWeight}, and $M$ is any parameter satisfying
\begin{align}\label{eq: definitionOfM}
   M\geq \frac{\norm{u}{L^\infty([0,t]\times\R)}+\norm{v}{L^\infty([0,t]\times\R)}}{2}.
\end{align}
Thus, there is at most one entropy solution of \eqref{eq: theMainEquation} for each $u_0\in L^\infty(\R)$.
\end{proposition}
\begin{proof}
 We begin by reformulating \eqref{eq: entropyInequality} in terms of the Kru\v{z}kov entropies; parameterized over $k\in\R$, they are given by $(\eta_k(u),q_k(u))=(|u-k|,F(u,k))$ where
 \begin{align*}
     F(u,k)\coloneqq \tfrac{1}{2}\sgn(u-k)(u^2-k^2).
 \end{align*}
These entropy pairs lack the required smoothness, but are still applicable in \eqref{eq: entropyInequality} as they can be smoothly approximated. Indeed, consider for $\d>0$ and $k\in\R$ the entropy pairs $\eta_{k}^{\d}(u)=\sqrt{(u-k)^2+\d^2}$ and $q_{k}^{\d}(u)=\int_k^u(\eta_{k}^{\d})'(y)y\mathrm{d}y$. As we have the pointwise limits
\begin{align*}
    \lim_{\d\to0}\eta^\d_k(u)=\,&|u-k|, &\lim_{\d\to0}q^\d_k(u)=\,&F(u,k), &\lim_{\d\to0}(\eta^\d_k)'(u)=\,&\sgn(u-k),
\end{align*}
we can substitute $(\eta,q)\mapsto(\eta^\d_k,q_k^\d)$ in \eqref{eq: entropyInequality} and let $\d\to0$ to conclude through dominated convergence that $u$ satisfies
\begin{align}\label{eq: kruzkovEntropyInequalities}
        0\leq\int_{0}^\infty\int_{\R} |u-k|\varphi_t+F(u,k)\varphi_x +\sgn(u-k) (K\ast u)\varphi \mathrm dx\mathrm dt,
\end{align}
for all $k\in\R$ and all non-negative $\varphi\in C_c^\infty(\R^+\times\R)$. For brevity, we set $U=\R^+\times \R$ for use throughout the proof. Let $\psi\in C^\infty_c(U\times U)$ be non-negative, and consider $u$ and $v$ as functions in $(t,x)$ and $(s,y)$ respectively. For fixed $(s,y)\in U$, we can in \eqref{eq: kruzkovEntropyInequalities} insert the test-function $\varphi\colon (t,x)\mapsto \psi(t,x,s,y)$ and the constant $k=v(s,y)$ so to obtain
\begin{equation}\label{eq: entropyInequalityForUWithVAsK}
    \begin{split}
             0\leq&\,\int_{U} |u-v|\psi_t+F(u,v)\psi_x + \sgn(u-v) (K\ast_x u)\psi \mathrm dx\mathrm dt,
    \end{split}
\end{equation}
where we write $K\ast_x u$ to stress that the operator $K\ast$ is applied  with respect to the $x$-variable. As \eqref{eq: entropyInequalityForUWithVAsK} holds for all $(s,y)\in U$ we can integrate \eqref{eq: entropyInequalityForUWithVAsK} over $(s,y)\in U$ giving
\begin{equation}\label{eq: integratedEntropyInequalityForUWithVAsK}
    \begin{split}
             0\leq&\,\int_U\int_{U} |u-v|\psi_t+F(u,v)\psi_x+\sgn(u-v) (K\ast_x u)\psi \mathrm dx\mathrm dt\mathrm dy\mathrm ds.
    \end{split}
\end{equation}
Swapping the roles of $u(t,x)$ and $v(s,y)$ we similarly find
\begin{equation}\label{eq: integratedEntropyInequalityForVWithUAsK}
    \begin{split}
             0\leq&\,\int_U\int_{U} |u-v|\psi_s+F(v,u)\psi_y+\sgn(v-u) (K\ast_y v)\psi \mathrm dx\mathrm dt\mathrm dy\mathrm ds.
    \end{split}
\end{equation}
As $F(u,v)=F(v,u)$ and $\sgn(v-u)=-\sgn(u-v)$ we can add \eqref{eq: integratedEntropyInequalityForUWithVAsK} to \eqref{eq: integratedEntropyInequalityForVWithUAsK} so to further obtain
\begin{equation}\label{eq: integrateCombineddEntropyInequalityWithUAndV}
    \begin{split}
             0\leq&\,\int_U\int_{U} |u-v|(\psi_t+\psi_s)+F(u,v)(\psi_x+\psi_y)\mathrm dx\mathrm dt\mathrm dy\mathrm ds\\
             &\,+\int_U\int_U\sgn(u-v) (K\ast_x u - K\ast_y v)\psi \mathrm dx\mathrm dt\mathrm dy\mathrm ds.
    \end{split}
\end{equation}
Let $\rho\in C^\infty_c(\R^2)$ be non-negative and satisfy $\norm{\rho}{L^1(\R^2)} = 1$, and let $\rho_\e$ denote the expression
\begin{align*}
    \rho_\e = \rho_\e(t-s,x-y) = \frac{1}{\e^2}\rho\bigg(\frac{t-s}{\e},\frac{x-y}{\e}\bigg),
\end{align*}
for $\e>0$. For a fixed $T\in (0,\infty)$, we further let $\varphi$ denote a non-negative element of $ C^\infty_c((0,T)\times\R)$ and set
\begin{align*}
    \psi(t,x,s,y)=\varphi(t,x)\rho_\e(t-s,x-y),
\end{align*}
or simply $\psi=\varphi\rho_\e$ for short. Note that, for $\e>0$ sufficiently small, this $\psi$ is non-negative, smooth and of compact support in $U\times U$; in particular, it satisfies the prior assumptions posed on it. Using that $(\partial_t+\partial_s)\rho_\e=0=(\partial_x+\partial_y)\rho_\e$, we can conclude
\begin{align*}
    (\psi_t+\psi_s)=&\,  \varphi_t\rho_\e, & (\psi_x+\psi_y)=&\, \varphi_x\rho_\e ,
\end{align*}
and so inserting for $\psi$ in \eqref{eq: integrateCombineddEntropyInequalityWithUAndV} we get
\begin{equation}\label{eq: integrateCombineddEntropyInequalitySpecialTestFunction}
    \begin{split}
             0\leq&\,\int_U\int_{U} \Big[|u-v|\varphi_t+F(u,v)\varphi_x\Big]\rho_\e\mathrm dx\mathrm dt\mathrm dy\mathrm ds\\
             &\,+\int_U\int_U\sgn(u-v) (K\ast_x u - K\ast_y v)\varphi\rho_\e \mathrm dx\mathrm dt\mathrm dy\mathrm ds.
    \end{split}
\end{equation}
We now wish to `go to the diagonal' by taking $\limsup_{\e\to0}$ of \eqref{eq: integrateCombineddEntropyInequalitySpecialTestFunction}; for simplicity we study each line separately.
For the first one we pick $M\in(0,\infty)$ satisfying the inequality \eqref{eq: definitionOfM} with $T$ replacing $t$,
and use $(u^2-v^2)=(u+v)(u-v)$ to compute
\begin{equation}\label{eq: goingToTheDiagonalFirstPart}
    \begin{split}
         &\,\int_U\int_{U} \Big[|u-v|\varphi_t+F(u,v)\varphi_x\Big]\rho_\e\mathrm dx\mathrm dt\mathrm dy\mathrm ds\\
    \leq &\,\int_U\int_{U} |u-v|\Big[\varphi_t+M|\varphi_x|\Big]\rho_\e\mathrm dx\mathrm dt\mathrm dy\mathrm ds\\
    \leq &\,\int_U |u(t,x)-v(t,x)|\Big[\varphi_t+M|\varphi_x|\Big]\mathrm dx\mathrm dt\\
    &\, + \int_U\int_{U} |v(t,x)-v(s,y)| \Big[\varphi_t+M|\varphi_x|\Big] \rho_\e \mathrm dx\mathrm dt\mathrm dy\mathrm ds.
    \end{split}
\end{equation}
As $\rho_\e(t-s,x-y)$ is supported in the region $|(t-s,x-y)|\leq \e$ and satisfies $\norm{\rho_\e}{L^1(\R^2)}=1$, the very last integral in \eqref{eq: goingToTheDiagonalFirstPart} is bounded by
\begin{align*}
    \sup_{|(\epsilon,\delta)|\leq \e}\int_U |v(t,x)-v(t+\epsilon, x+ \delta)|\Big[\varphi_t+M|\varphi_x|\Big] \mathrm{d}x\dd t \to 0,\quad \e\to0,
\end{align*}
where the limit holds as translation is a continuous operation on $L^1_{\loc}(\R)$ and $\varphi\in C^\infty_c((0,T)\times \R)$. Thus we have established
\begin{equation}\label{eq: goingToTheDiagonalLimitOfFirstPart}
    \begin{split}
         &\,\limsup_{\e\to0}\int_U\int_{U} \Big[|u-v|\varphi_t+F(u,v)\varphi_x\Big]\rho_\e\mathrm dx\mathrm dt\mathrm dy\mathrm ds\\
    \leq &\,\int_U |u(t,x)-v(t,x)|\Big[\varphi_t+M|\varphi_x|\Big]\mathrm dx\mathrm dt.
    \end{split}
\end{equation}
Turning our attention to the second line of \eqref{eq: integrateCombineddEntropyInequalitySpecialTestFunction}, we start by observing
\begin{equation*}
    \begin{split}
        &\,\int_U\int_U\sgn(u-v) (K\ast_x u - K\ast_y v)\varphi\rho_\e \mathrm dx\mathrm dt\mathrm dy\mathrm ds\\
        \leq &\,\int_U\int_U\int_\R |K(z)||u(t,x-z) - v(s,y-z)|\varphi(t,x) \rho_\e(t-s,x-y)\mathrm dz \mathrm dx\mathrm dt\mathrm dy\mathrm ds\\
        = &\,\int_U\int_U\int_\R |K(z)||u(t,x) - v(s,y)|\varphi(t,x+z)\rho_\e(t-s,x-y) \mathrm dz \mathrm dx\mathrm dt\mathrm dy\mathrm ds\\
        = &\, \int_U\int_U |u - v|\Big[|K|\ast_x\varphi\Big]\rho_\e \mathrm dx\mathrm dt\mathrm dy\mathrm ds,
    \end{split}
\end{equation*}
where the third line holds by the substitution $(x,y)\mapsto(x+z,y+z)$ and the last by the symmetry of $z\mapsto|K(z)|$. By similar reasoning used to attain \eqref{eq: goingToTheDiagonalFirstPart}, we conclude
\begin{equation}\label{eq: goingToTheDiagonalLimitOfSecondPart}
    \begin{split}
        &\,\limsup_{\e\to0}\int_U\int_U\sgn(u-v) (K\ast_x u - K\ast_y v)\rho_\e\varphi \mathrm dx\mathrm dt\mathrm dy\mathrm ds\\
        \leq&\,\int_U|u(t,x)-v(t,x)|(|K|\ast\varphi) \mathrm dx\mathrm dt,
    \end{split}
\end{equation}
Combining \eqref{eq: integrateCombineddEntropyInequalitySpecialTestFunction} with \eqref{eq: goingToTheDiagonalLimitOfFirstPart} and \eqref{eq: goingToTheDiagonalLimitOfSecondPart}, yields the inequality
\begin{align}\label{eq: katoInequality}
    0\leq \int_U|u-v|\Big[\varphi_t+M|\varphi_x| + |K|\ast\varphi\Big] \mathrm dx\mathrm dt,
\end{align}
where both $u$ and $v$ are now functions in $(t,x)$. By density, we may extend \eqref{eq: katoInequality} to hold for all non-negative $\varphi\in W^{1,1}_0((0,T)\times\R)$. Thus, we can set $\varphi(t,x)=\theta(t)\phi(t,x)$ for two non-negative functions $\theta\in W^{1,1}_0((0,T))$ and $\phi\in W^{1,1}((0,T)\times \R)$ where we note that $\phi$ need not vanish at $t=0$ and $t=T$. In doing so, \eqref{eq: katoInequality} yields
\begin{align}\label{eq: katoInequalityTwo}
    0\leq \int_U|u-v|\theta'\phi\dd x\dd t + \int_U|u-v|\theta\Big[\phi_t+M|\phi_x| + |K|\ast\phi\Big] \mathrm dx\mathrm dt,
\end{align}
To rid ourselves of the second integral, we now construct a particular $\phi$ such that the square bracket in \eqref{eq: katoInequalityTwo} is non-positive in $(0,T)\times \R$. Let $f\colon\R\to[0,1]$ be smooth, non-increasing and satisfy $f(x)=1$ for $x\leq 0$ and $f(x)=0$ for sufficiently large $x$, and define
\begin{align}\label{eq: definitionOfG}
    g(t,x)=f(|x| + M (t-T)).
\end{align}
It is readily checked that $g\in C^\infty_c([0,T]\times\R)$.
We now define the function $\phi$ to be
\begin{equation}\label{eq: definitionOfPhi}
    \begin{split}
           \phi(t,x) =&\, \Big(e^{(T-t)|K|\ast}g(t,\cdot)\Big)(x),
    \end{split}
\end{equation}
where we used the operator defined in \eqref{eq: definitionOfExponentialConvolutionOperator}. Observe that $\phi$ is non-negative and smooth on $[0,T]\times \R$ with integrable derivatives; this last part follows when using \eqref{eq: boundnessOfExponentialConvolutionOperator}. That the square bracket in \eqref{eq: katoInequalityTwo} is non-positive, can be seen as follows: note first from \eqref{eq: definitionOfG} that
\begin{align*}
    g_t(t,x)=&\,M f'(|x|+M (t-T)),\\
    g_x(t,x)=&\, \sgn(x)f'(|x|+M (t-T)).
\end{align*}
As $f'$ is non-positive, we find $g_t=-M |g_x|$. Thus, using \eqref{eq: definitionOfPhi} we calculate for $t\in(0,T)$
\begin{align*}
    \phi_t + |K|\ast \phi =&\,  e^{(T-t)|K|\ast}g_t, \\
    =&\, - M\Big( e^{(T-t)|K|\ast}|g_x|\Big),\\
    \leq&\,  - M \Big|e^{(T-t)|K|\ast}g_x\Big|\\
    =&\,- M |\phi_x|,
\end{align*}
where the last equality holds as differentiation commutes with convolution. In conclusion, the second integral in \eqref{eq: katoInequalityTwo} is non-positive. Next, for a small parameter $\epsilon>0$ we set $\theta=\theta_\epsilon$ where $\theta_\epsilon$ is given by
\begin{equation}\label{eq: thetaWhichApproximatesInterval}
    \begin{split}
           \theta_\epsilon(t)=
    \begin{cases}
    t/\epsilon, &t\in(0,\epsilon),\\
    1, &t\in(\epsilon,T-\epsilon),\\
    (T-t)/\epsilon, &t\in(T-\epsilon,T).
    \end{cases}
    \end{split}
\end{equation} 
Inserting this in \eqref{eq: katoInequalityTwo}, removing the non-positive integral and letting $\epsilon\to0$, we conclude
\begin{equation}\label{eq: L1ContractionOnAverage}
    \begin{split}
        &\,\liminf_{\epsilon\to0}\int_{T-\epsilon}^T\bigg( \int_\R|u(t,x)-v(t,x)| \phi(t,x) \mathrm dx\bigg)\frac{\mathrm dt}{\epsilon}\\
        \leq&\, \limsup_{\epsilon\to0}\int_{0}^\epsilon\bigg( \int_\R|u(t,x)-v(t,x)| \phi(t,x) \mathrm dx\bigg)\frac{\mathrm dt}{\epsilon}\\
    \end{split}
\end{equation}
where we moved the negative term over to the left-hand side. As $u$ and $v$ are bounded on $(0,T)\times \R$ and continuous at $t=0$ in $L^1_\loc$ sense, it is easy to see that $|u(t,\cdot)-v(t,\cdot)|\phi(t,\cdot)\to |u_0(\cdot)-v_0(\cdot)|\phi(0,\cdot)$ in $L^1(\R)$ when $t\to0$ since the same is true for $\phi(t,x)$ and $\phi(0,x)$. Thus the right-hand side of \eqref{eq: L1ContractionOnAverage} is given by
\begin{align*}
    \limsup_{\epsilon\to0}\int_{0}^\epsilon\bigg( \int_\R|u(t,x)-v(t,x)| \phi(t,x) \mathrm dx\bigg)\frac{\mathrm dt}{\epsilon} = \int_{\R}|u_0-v_0|\phi(0,x)\dd x.
\end{align*}
 As for the left-hand side, we wish to apply the Lebesgue differentiation theorem so to get convergence for a.e. $T>0$, but this can not be directly done due to the implicit $T$-dependence of $\phi$. Instead, we observe from \eqref{eq: definitionOfG} and \eqref{eq: definitionOfPhi} that $\phi(T,x)=g(T,x)=f(|x|)$ where the latter function is independent of $T$. Since $\varphi(t,\cdot)\to f(|\cdot|)$ in $L^1(\R)$ as $t\to T$, the boundness of $u$ and $v$ means that $|u(t,\cdot)-v(t,\cdot)|(\varphi(t,\cdot)-f(|\cdot|))\to0$ in $L^1(\R)$ as $t\to T$ and so we may estimate
\begin{align*}
      &\,\limsup_{\epsilon\to0}\int_{T-\epsilon}^{T}\bigg( \int_\R|u(t,x)-v(t,x)| \phi(t,x)\mathrm dx\bigg)\frac{\mathrm dt}{\epsilon}\\
      =&\,\limsup_{\epsilon\to0}\int_{T-\epsilon}^{T}\bigg( \int_{\R}|u(t,x)-v(t,x)|f(|x|) \mathrm dx\bigg)\frac{\mathrm dt}{\epsilon}\\
      =&\, \int_{\R}|u(T,x)-v(T,x)| f(|x|)\mathrm dx,\quad \text{a.e. } T\geq 0,
\end{align*}
where the last equality used the Lebesgue differentiation theorem. Thus we conclude from \eqref{eq: L1ContractionOnAverage} that we for a.e. $T\geq 0$ have
\begin{equation}\label{eq: atTheEndOfL1Contraction}
    \begin{split}
          &\,\int_{\R}|u(T,x)-v(T,x)|f(|x|)\dd x\\
    \leq&\, \int_{\R}|u_0(x)-v_0(x)|\Big(e^{T|K|\ast}f(|\cdot| - MT)\Big)(x)\dd x,
    \end{split}
\end{equation}
where we inserted for $\phi(0,x)$ using \eqref{eq: definitionOfG} and \eqref{eq: definitionOfPhi}. As $f$ was any smooth, non-negative, non-increasing function satisfying $f(x)=1$ for $x\leq0$ and $f(x)=0$ for sufficiently large $x$, we may in \eqref{eq: atTheEndOfL1Contraction} set  $f=\mathbbm{1}_{(-\infty,r)}$ through a standard approximation argument. Doing this, we observe that $f(|x|-MT) =\chi_M^r(T,x)$ where the latter is defined in \eqref{eq: definitionOfTheCharacteristicFunction}, and so we obtain from \eqref{eq: atTheEndOfL1Contraction} exactly \eqref{eq: fullL1Contraction},  with $T$ substituting for $t$. This concludes the proof.
\end{proof}
While we in this paper are concerned with global entropy solutions, one may wish to study entropy solutions on a time-bounded domain $(0,T)\times \R$. Such solutions would be defined as in Def. \ref{def: entropySolution}, but with the test-functions in \eqref{eq: entropyInequality} restricted to $C^\infty_c((0,T)\times\R)$. Still, no new solutions are attained this way: the uniqueness of entropy solutions on a time-bounded domain follows from the same argument as above, and thus an entropy solution on $(0,T)\times \R$ is the restriction of a global one which the following section establishes the existence of.

\subsection{Existence of entropy solutions}\label{sec: existenceOfEntropySolutions}
In this subsection, we prove the existence of an entropy solution of \eqref{eq: theMainEquation} for arbitrary initial data $u_0\in L^\infty(\R)$.
The strategy goes as follows: we first introduce for a parameter $\e>0$ an approximate solution map $S_{\e,t}\colon L^\infty(\R)\to L^\infty(\R)$ whose key properties are collected in Proposition \ref{prop: propertiesOfTheApproximateSolutionMap}. Next, we show in Lemma \ref{lem: approximateSolutionSatisfiesApproximateEntropySolution} that when $S_{\e,t}$ is applied to sufficiently regular initial data $u_0$, we attain approximate entropy solutions. Further, in Proposition \ref{prop: solutionWithBVDataIsLimitOfApproximateSolutions} we establish the convergence (as $\e\to0$) of these approximations to an entropy solution, and the result is extended to general $L^\infty$ data in Corollary \ref{cor: generalSolutionsAreLimitsOfBVSolutions}.

By an operator splitting argument, we aim to build entropy solutions of \eqref{eq: theMainEquation} from those of Burgers' equation, $u_t+\tfrac{1}{2}(u^2)_x=0$, and the linear convolution equation, $u_t= K\ast u$. On that note, we introduce two families of operators $(S_t^B)_{t\geq0}$ and $(S^K_t)_{t\geq0}$ parameterized over $t\geq 0$. The operator $S^B_t\colon L^\infty(\R)\to L^\infty(\R)$ is the solution map for Burgers' equation restricted to $L^\infty$ data at time $t$; that is,
\begin{align}\label{eq: definitionOfSolutionMapOfBurgersEquation}
   S^B_t\colon f \mapsto u^f(t,\cdot),
\end{align}
where $(t,x)\mapsto u^f(t,x)$ is the unique bounded entropy solution for the problem
\begin{align*}
    \begin{cases}
    u_t+\tfrac{1}{2}(u^2)_x = 0,  & (t,x)\in\R^+\times\R,\\
    u(0,x)=f(x),  &x\in\R.
    \end{cases}
\end{align*}
As demonstrated in \cite{MR2574377}, this solution lies in $C([0,\infty), L^1_{\text{loc}}(\R))$, the space of functions $u\in L^1_\loc([0,\infty)\times \R)$ such that $t\mapsto u(t,\cdot)$ is continuous from $[0,\infty)$ to $L^1_\loc(\R)$.
Note that $S^B_t$ is a flow map in the sense that $S^B_{t_1}\circ S^B_{t_2}=S^B_{t_1+t_2}$ for all $t_1,t_2\geq 0$. The second map $S^K_t\colon L^\infty(\R)\to L^\infty(\R)$ is for $t\geq0$ defined by
\begin{align}\label{eq: definitionOfSolutionMapOfConvolutionEquation}
    S^K_t\colon f \mapsto f+ tK\ast f.
\end{align}
The actual solution map for the equation $u_t= K\ast u$ is the operator $e^{t K\ast}$ defined similarly to \eqref{eq: definitionOfExponentialConvolutionOperator}; the reason we have instead chosen $S^K_t$ as \eqref{eq: definitionOfSolutionMapOfConvolutionEquation} (which can be seen as a first order approximation of $e^{tK\ast}$) is for our calculations to be slightly tidier. Note however, $S^K_t$ is not a flow mapping. With these two families of operators, we build a third family of operators $S_{\e,t}$: for fixed parameters $\e>0$ and $t\geq 0$, pick the unique pair $n\in\N_0$ and $s\in[0,\e)$ such that $t=s+n\e$, and define
\begin{align}\label{eq: definitionOfApproximationOfSolutionMap}
    S_{\e,t} =&\, S^B_s\circ \Big[ S^K_\e\circ S^B_\e\Big]^{\circ n},
\end{align}
where the notation $\circ n$ implies that the square bracket is composed with itself $(n-1)$ times; if $n=0$, then the square bracket should be replaced by the identity. We shall demonstrate that as $\e\to0$ the map $S_{\e,t}$ converges in an appropriate sense to the solution map for entropy solutions of \eqref{eq: theMainEquation}. We begin by collecting a few properties of $S_{\e,t}$ when applied to the space $BV(\R)$; this subspace of $L^1(\R)$ is equipped with the norm $\norm{\cdot}{BV(\R)}=\norm{\cdot}{L^1(\R)}+|\cdot|_{TV}$, where the total variation seminorm $|\cdot|_{TV}$ coincides with $|\cdot|_{TV^1}$ as defined in \eqref{eq: definitionOfFractionalVariance}. A short and effective discussion of $BV(\R)$ can be found in either \cite{MR2574377} or \cite{HoldenRisebro}; we note that functions in $BV(\R)$ have essential right and left limits at each point, and their height is bounded by their total variation, thus $BV(\R)\hookrightarrow L^1\cap L^\infty(\R)$.
\begin{proposition}\label{prop: propertiesOfTheApproximateSolutionMap}
With $S_{\e,t}$ as defined in \eqref{eq: definitionOfApproximationOfSolutionMap}, we have for all $\e>0$, $t\geq\tilde t\geq 0$, $f\in BV(\R)$ and $p\in [1,\infty]$
\begin{align*}
      \norm{S_{\e,t}(f)}{L^p(\R)}\leq&\, e^{t\kappa}\norm{f}{L^p(\R)}, & &\,(\text{$L^p$ bound}),\\
     \norm{S_{\e,t}(f)}{TV}\leq&\, e^{t\kappa}\norm{f}{TV},& &\,(\text{$TV$ bound}),\\
     \norm{S_{\e,t}(f)-S_{\e,\tilde t}(f)}{L^1(\R)}\leq&\, (t-\tilde t + \e)C_f(t) , & &\,(\text{Approximate time continuity}),
\end{align*}
 where $\kappa\coloneqq \norm{K}{L^1(\R)}$ and where the factor $C_f(t)$ only depends on $f$ and $t$.
\end{proposition}
\begin{proof}
Consider $\e>0$ fixed. We will be using the following properties of the mappings $S^B_t$ and $S^K_t$ 
\begin{align}\label{eq: propertyOneOfBurgerMap}
    \norm{S^B_t(f)}{L^p(\R)}\leq&\, \norm{f}{L^p(\R)}, & \norm{S^K_t(f)}{L^p(\R)}\leq&\, e^{t\k}\norm{f}{L^p(\R)},\\\label{eq: propertyThreeOfBurgerMap}
     |S^B_t(f)|_{TV}\leq&\, |f|_{TV}, & |S^K_t(f)|_{TV}\leq&\, e^{t\k}|f|_{TV},\\\label{eq: propertyFourOfBurgerMap}
     \norm{S^B_t(f)-f}{L^1(\R)}\leq&\, t|f|_{TV}^2, & \norm{S^K_t(f)-f}{L^1(\R)}\leq&\, t\k\norm{f}{L^1(\R)},
\end{align}
valid for all $t\geq0$, $p\in[1,\infty]$ and $f\in BV(\R)$. The inequalities involving $S^B_t$ are well known and can be found for example in \cite{HoldenRisebro}. As for the inequalities involving $S^K_t$, these estimates follow directly from the definition of $S^K_t$ \eqref{eq: definitionOfSolutionMapOfConvolutionEquation} together with Young's convolution inequality and $1+t\k\leq e^{t\k}$. We start by proving the $L^p$ and $TV$ bound of the proposition. For this we fix $t\geq0$ and pick $n\in \N_0$ and $s\in[0,\e)$ such that $t=s+n\e$, and pick an arbitrary $f\in BV(\R)$. By iteration of the two inequalities in \eqref{eq: propertyOneOfBurgerMap} we attain
\begin{align}\label{eq: L1BoundForFTilde}
     \norm{S_{\e,t} (f)}{L^p(\R)}=\norm{S_{s}^{B}\circ[S^K_\e\circ S^B_\e]^{\circ n} (f)}{L^p(\R)}\leq e^{ n\e\kappa}\norm{f}{L^p(\R)},
\end{align}
for all $p\in[1,\infty]$, and by iteration of the inequalities in \eqref{eq: propertyThreeOfBurgerMap} we similarly get
\begin{align}  \label{eq: TVBoundForFTilde}
    | S_{\e,t}(f)|_{TV}=|S_{s}^{B}\circ[S^K_\e\circ S^B_\e]^{\circ n} (f)|_{TV}\leq e^{ n\e\kappa}|f|_{TV}.
\end{align}
This gives the first two bounds of the proposition. For the time continuity, we pick $\tilde t\in[0,t]$ and $\tilde n\in \N$ and $\tilde s\in[0,\e)$ such that $\tilde t=\tilde s+\tilde n\e$. Suppose first that $t-\tilde t\leq \e$, and set $\tilde f = S_{\e,\tilde n\e}(f)$. Then either $S_{\e,t}(f)= S_{s-\tilde s}^B(\tilde f)$ or $S_{\e,t}(f)= S^B_s\circ S^K_\e\circ S_{\e-\tilde s}^B(\tilde f)$ corresponding to the two situations $n=\tilde n$ and $n=\tilde n+1$; we will only deal with the latter as the other case is dealt with similarly. By the triangle inequality we then have
\begin{align*}
     \norm{S_{\e,t}(f)-S_{\e,\tilde t}(f)}{L^1(\R)}
    \leq &\, \norm{S^B_s\circ S_{\e}^K\circ S^B_{\e-\tilde s}(\tilde f)-S_{\e}^K\circ S^B_{\e-\tilde s}(\tilde f)}{L^1(\R)}\\
    &\,+\norm{S_{\e}^K\circ S^B_{\e-\tilde s}(\tilde f)-S^B_{\e-\tilde s}(\tilde f)}{L^1(\R)} + \norm{ S^B_{\e-\tilde s}(\tilde f)-\tilde f}{L^1(\R)}.
\end{align*}
The three terms on the right-hand side can be directly dealt with using the two inequalities \eqref{eq: propertyFourOfBurgerMap} followed by the estimates \eqref{eq: L1BoundForFTilde} and \eqref{eq: TVBoundForFTilde}. Doing so in a straight forward manner results in the bound
\begin{align*}
    se^{2n\e\k}|f|_{TV}^2 + \e\k e^{\tilde n\e\k}\norm{f}{L^1(\R)} + (\e-\tilde s)e^{2\tilde n\e\k}|f|_{TV}^2  \leq \e e^{2t\k}(2|f|_{TV}^2+\k\norm{f}{L^1(\R)}).
\end{align*}
Thus, setting for example $C_f(t)=e^{2t\k}(2|f|_{TV}^2+\k\norm{f}{L^1(\R)})$ the time continuity estimate holds whenever $t-\tilde t\leq \e$. By breaking any large time step into steps of size no larger than $\e$, the general case follows by the triangle inequality. 
\end{proof}
The $L^p$ bound provided by the previous proposition was attained by applying Young's convolution inequality on the operator $K\ast$; in doing so, we miss possible cancellations that might take place as $K$, after all, is an odd function. While efficient $L^p$ bounds might not be feasible for general $p\geq 1$, these cancellations are easily exploited for the $L^2$ norm as seen from the following lemma. This $L^2$ control is crucial for the analysis of Section \ref{sec: longTermBounds}.
\begin{lemma}\label{lem: theApproximateSolutionMapConservesTheL2Norm}
With $S_{\e,t}$ as defined in \eqref{eq: definitionOfApproximationOfSolutionMap}, we have for all $\e>0$, $t\geq 0$ and $f\in L^2\cap L^\infty(\R)$
\begin{align*}
    \norm{S_{\e,t}(f)}{L^2(\R)}\leq&\, e^{\frac{1}{2}\e t \kappa^2}\norm{f}{L^2(\R)},
\end{align*}
where $\kappa\coloneqq \norm{K}{L^1(\R)}$.
\end{lemma}
\begin{proof}
Consider $\e>0$ and $t\geq 0$ fixed. As $K$ is odd, real valued and in $L^1(\R)$, it is readily checked that $K\ast$ is a skew-symmetric operator on $L^2(\R)$, and consequently $\cop{f,K\ast f}=0$ for all $f\in L^2(\R)$.
In particular,
\begin{align*}
    \norm{S^{K}_\e(f)}{L^2(\R)}^2
    = \cop{f,f} +\e^2 \cop{K\ast f, K\ast f}
    \leq &\, (1+\e^2\k^2)\norm{f}{L^2(\R)}^2.
\end{align*}
Combined with $1+\e^2\k^2\leq e^{\e^2\k^2}$ and the fact that $\norm{S^B_\e (f)}{L^2(\R)}\leq\norm{f}{L^2(\R)}$ (left-most inequality in \eqref{eq: propertyOneOfBurgerMap}), the result follows by iteration.
\end{proof}
When $u_0\in BV(\R)$, we can use $S_{\e,t}$ to construct a family of approximate entropy solutions of \eqref{eq: theMainEquation} as follows. For an arbitrary, but fixed,  $u_0\in BV(\R)$, let the family $(u^\e)_{\e>0}\subset L^\infty_{\text{loc}}([0,\infty),L^\infty(\R))$ be defined by
\begin{align}\label{eq: definitionOfFamilyOfApproximateSolutions}
    u^\e(t)=S_{\e,t}(u_0),
\end{align}
where $u^\e(t)$ is compact notation for $x\mapsto u^\e(t,x)$. Although $(u^\e)_{\e>0}$ is considered a family in $L^\infty_{\text{loc}}([0,\infty),L^\infty(\R))$, we stress that each member is for all $t\geq0$ well defined in $L^\infty(\R)$. For small $\e>0$ these functions are not far off from satisfying the entropy inequality \eqref{eq: entropyInequality}, as we now show.
\begin{lemma}\label{lem: approximateSolutionSatisfiesApproximateEntropySolution}
With $(u^\e)_{\e>0}$ as defined in \eqref{eq: definitionOfFamilyOfApproximateSolutions} for some $u_0\in BV(\R)$, we have for every entropy pair $(\eta,q)$ of \eqref{eq: theMainEquation} and non-negative $\varphi\in C^\infty_c(\R^+\times\R)$ the approximate entropy inequality
\begin{equation*}
    \begin{split}
    &\,\int_{0}^\infty \int_{\R}\eta(u^{\e})\varphi_t +q(u^{\e})\varphi_x +\eta'(u^\e)(K\ast u^\e)\varphi\mathrm{d}x\mathrm{d}t
    \geq  O(\e).
    \end{split}
\end{equation*}
\end{lemma}
\begin{proof}
Fixing $\e>0$, we observe from the definition of $S_{\e,t}$ \eqref{eq: definitionOfApproximationOfSolutionMap} that $u^\e$ is an entropy solution of Burgers' equation on the open sets $(t^\e_{n-1},t^\e_n)\times\R$ for $n\in\N$, where $t_n^\e=n\e$; thus
\begin{align}\label{eq: burgersEntropyEquation}
     \int_{t_{n-1}^\e}^{t_n^\e}\int_\R\eta(u^{\e})\varphi_t +q(u^{\e})\varphi_x \mathrm{d}x\mathrm{d}t\geq0,
\end{align}
for every non-negative $\varphi\in C_c^\infty((t^\e_{n-1},t^\e_n)\times\R)$ and every entropy pair $(\eta,q)$ of Burgers' equation, which coincides with the entropy pairs of \eqref{eq: theMainEquation} as the convection term of the two equations agree. Moreover, by the time continuity of $S_t^B$ \eqref{eq: propertyThreeOfBurgerMap} and the $TV$ bound from Proposition \ref{prop: propertiesOfTheApproximateSolutionMap}, we see that $u^\e\in C([t_{n-1}^\e,t_n^\e),L^1_{\loc}(\R))$; at $t=t_n^\e$ it is discontinuous from the left, as the left limit is given by $u^\e(t_n^\e-)=S^B_\e(u^\e(t_{n-1}^\e))$, while we have defined
\begin{align}\label{eq: correlationBetweenLeftAndRightLimitOfApproximateSolution}
    u^\e(t^\e_n)=u^\e(t^\e_n-) + \e K\ast u^{\e}(t^\e_n-).
\end{align}
The continuity in time allows us, by a similar trick used to attain \eqref{eq: L1ContractionOnAverage}, to extend \eqref{eq: burgersEntropyEquation} to
\begin{equation}\label{e: extendedBurgersEquation}
    \begin{split}
         \int_{t_{n-1}^\e}^{t_n^\e}\int_\R\eta(u^{\e})\varphi_t +q(u^{\e})\varphi_x \mathrm{d}x\mathrm{d}t\geq&\, \int_{\R}\eta(u^{\e}(t_n^\e-))\varphi(t_n^\e,x)\mathrm{d}x\\
    &\,- \int_{\R}\eta(u^{\e}(t_{n-1}^\e))\varphi(t_{n-1}^\e,x)\mathrm{d}x,
    \end{split}
\end{equation}
for all non-negative $\varphi\in C^\infty_c(\R^+\times\R)$. For the remainder of the proof, consider the entropy pair $(\eta,q)$ and $\varphi\in C^\infty_c(\R^+\times\R)$ fixed. Summing \eqref{e: extendedBurgersEquation} over $n\in\N$ and using $\varphi(0,x)=0$, we get
\begin{equation}\label{eq: almostEntropySolutionRaw}
    \begin{split}
    &\,\int_{\R^+\times\R}\eta(u^{\e})\varphi_t +q(u^{\e})\varphi_x \mathrm{d}x\mathrm{d}t\\
    \geq&\, \sum_{n=1}^\infty\int_{\R}\Big[\eta(u^{\e}(t_n^\e-))-\eta(u^{\e}(t_n^\e))\Big]\varphi(t_n^\e,x)\mathrm{d}x.
    \end{split}
\end{equation}
By Proposition \ref{prop: propertiesOfTheApproximateSolutionMap}, the family $(u^\e)_{\e>0}$ is uniformly bounded on the support of $\varphi$, and so we can assume without loss of generality that $|\eta'|,|\eta''|<C_1$ for some large $C_1$. Using the relation \eqref{eq: correlationBetweenLeftAndRightLimitOfApproximateSolution}, the square bracket from \eqref{eq: almostEntropySolutionRaw} can thus be estimated
\begin{equation*}
    \begin{split}
           &\,\eta(u^{\e}(t_n^\e-))-\eta(u^\e(t_n^\e))\\
   \geq &\,-\e\eta'(u^\e(t_n^{\e}-))\Big[K\ast u^\e(t_n^\e-)\Big] - \frac{C_1\e^2}{2}|K\ast u^\e(t^\e_n-)|^2,
    \end{split}
\end{equation*}
which, again by the uniform bound of $u^\e$ on the compact support of $\varphi$, further implies
\begin{equation}\label{eq: necessaryInequalityForApproximateSolution}
    \begin{split}
           &\,\int_{\R}\Big[\eta(u^{\e}(t_n^\e-))-\eta(u^\e(t_n^\e))\Big]\varphi(t_n^\e,x)\mathrm{d}x\\
   \geq &\,-\e\int_{\R}\eta'(u^\e(t_n^{\e}-))\Big[K\ast u^\e(t_n^\e-)\Big]\varphi(t_n,x)\mathrm{d}x - C_2\e^2,
    \end{split}
\end{equation}
for some $C_2>0$ independent of $n$ and $\e$.
Combining the uniform time regularity of Proposition \ref{prop: propertiesOfTheApproximateSolutionMap} and the compact support of $\varphi$, we see that the function
\begin{align}\label{eq: lipschitzOfTMappingToIntegral}
    g_\e(t)\coloneqq \int_\R\eta'(u^\e(t))\Big[K\ast u^\e(t)\Big]\varphi(t,x)\mathrm{d}x,
\end{align}
satisfies for all $t\geq \tilde t\geq 0$ an inequality $|g_\e(t)-g_\e(\tilde t)|\leq C_3(t-\tilde t + \e)$ for some sufficiently large $C_3$ independent of $\e$. Thus, the integral on the right-hand side of \eqref{eq: necessaryInequalityForApproximateSolution} can be bounded from below as such
\begin{equation}\label{eq: secondNecessaryInequalityForApproximateSolution}
    \begin{split}
    &\,-\e\int_{\R}\eta'(u^\e(t_n^{\e}-))\Big[K\ast u^\e(t_n^\e-)\Big]\varphi(t_n,x)\mathrm{d}x\\
    =&\,-\int_{t^{\e}_{n-1}}^{t^\e_n}\int_{\R}\eta'(u^\e(t_n^{\e}-))\Big[K\ast u^\e(t_n^\e-)\Big]\varphi(t_n,x)\mathrm{d}x\mathrm{d}t\\
    \geq &\,-\int_{t^{\e}_{n-1}}^{t^\e_n}\int_{\R}\eta'(u^\e(t))\Big[K\ast u^\e(t)\Big]\varphi(t,x)\mathrm{d}x\mathrm{d}t - 2C_3\e^2.
    \end{split}
\end{equation}
Picking the smallest $N(\e)\in\N$ such that $\supp\varphi \cap (\e N(\e),\infty)\times\R=\emptyset$, we combine \eqref{eq: almostEntropySolutionRaw}, \eqref{eq: necessaryInequalityForApproximateSolution} and \eqref{eq: secondNecessaryInequalityForApproximateSolution} to deduce
\begin{align*}
    &\,\int_{\R^+\times\R}\eta(u^{\e})\varphi_t +q(u^{\e})\varphi_x +\eta'(u^\e)(K\ast u^\e)\varphi\mathrm{d}x\mathrm{d}t
    \geq CN(\e)\e^2,
\end{align*}
for some sufficiently large $C>0$. And as $N(\e)\e^2 \sim \e$ the proof is complete.
\end{proof}

With the previous result at hand, it is natural to look for a limit function of $(u^\e)_{\e>0}$ as $\e\to0$; this would be a suitable candidate for an entropy solution of \eqref{eq: theMainEquation} with initial data $u_0\in BV(\R)$. In the next proposition, we do exactly this and collect a few properties about the resulting solution.
\begin{proposition}\label{prop: solutionWithBVDataIsLimitOfApproximateSolutions}
For any initial data $u_0\in BV(\R)$, let $(u^\e)_{\e>0}$ be as defined in \eqref{eq: definitionOfFamilyOfApproximateSolutions}. Then, for all $t\geq 0$ the following limit holds in $L^1_{\mathrm{loc}}(\R)$
\begin{align}\label{eq: solutionWithBVDataIsLimitOfApproximateSolutions}
    u^\e(t)\to u(t),\quad \e\to0,
\end{align}
where $u$ is an entropy solution of \eqref{eq: theMainEquation} with initial data $u_0$. Moreover, $u$ is an element of  $C([0,\infty), L^1(\R))\cap L^\infty_{\mathrm{loc}}([0,\infty),L^\infty(\R))$ and satisfies for all $t\geq0$
\begin{align}\label{eq: coarseSupBoundOfSolutionWithBVData}
          \norm{u(t)}{L^\infty(\R)}\leq&\, e^{t\k}\norm{u_0}{L^\infty(\R)},\\
\label{eq: L2BoundOfEntropySolutionWithBVData}
    \norm{u(t)}{L^2(\R)}\leq&\,\norm{u_0}{L^2(\R)},
\end{align}
where $\k\coloneqq\norm{K}{L^1(\R)}$.
\end{proposition}
\begin{proof}
We first prove the limit \eqref{eq: solutionWithBVDataIsLimitOfApproximateSolutions} for a special subsequence of $(u^\e)_{\e>0}$ and then generalize afterwards.
Fixing $t\geq 0$, we see from Proposition \ref{prop: propertiesOfTheApproximateSolutionMap} that the functions $(u^\e(t))_{\e>0}$ satisfy for any $p\in[1,\infty]$
\begin{align}\label{eq: uniformLpBoundUsedInProofOfExistenceForBVData}
    \norm{u^\e(t)}{L^p(\R)}\leq e^{t\kappa}\norm{u_0}{L^p(\R)},
\end{align}
and in particular, they are uniformly bounded in $L^1(\R)$. Moreover, they are equicontinuous with respect to translation
\begin{align*}
        \norm{u^\e(t,\cdot+h)-u^\e(t,\cdot)}{L^1(\R)}\leq h e^{t\kappa}|u_0|_{TV},
\end{align*}
for all $h>0$, and so by the Kolmogorov--Riesz compactness Theorem, any infinite subset of $(u^\e(t))_{\e>0}$ is relatively compact in $L^1_{\text{loc}}(\R)$; as we have skipped developing a tightness estimate for $(u^\e(t))_{\e>0}$, we can not claim the family to be relatively compact in $L^1(\R)$. The family $(u^\e)_{\e>0}$ is not equicontinuous in time and so we can not directly apply the Arzelà-Ascoli theorem, however, the family is for small $\e$ arbitrary close to be equicontinuous and so the proof of the theorem is still applicable; for clarity we perform the steps. By a standard diagonalization argument, we can select a sub-sequence $(u^{\e_{j}})_{j\in\N}\subset(u^\e)_{\e>0}$ such that $\lim_{j\to\infty}\e_j =0$ and  $u^{\e_j}(t)$ converges in $L^1_{\loc}(\R)$ for every $t\in E$ with $E$ being a countable dense subset of $\R^+$. Next, we claim that $u^{\e_j}(t)$ converges in $L^1_{\loc}(\R)$ for every $t\geq 0$. Indeed, fix $r>0$ for locality and pick $s\in E$ such that $|s-t|<\epsilon$ for some arbitrary $\epsilon>0$. By the time regularity estimate of Proposition \ref{prop: propertiesOfTheApproximateSolutionMap}, we have
\begin{align*}
    &\, \limsup_{j,i\to\infty}\int_{-r}^r|u^{\e_j}(t)-u^{\e_i}(t)|\mathrm{d}x\\
    \leq&\, \limsup_{j,i\to\infty}\int_{-r}^r|u^{\e_j}(t)-u^{\e_j}(s)| +|u^{\e_j}(s)-u^{\e_i}(s)|+|u^{\e_i}(s)-u^{\e_i}(t)|\mathrm{d}x\\
    \leq&\, \limsup_{j,i\to\infty}(2\epsilon +\e_j+\e_i)C_{u_0}(t+\epsilon)+\limsup_{j,i\to\infty}\int_{-r}^r|u^{\e_j}(s)-u^{\e_i}(s)|\mathrm{d}x\\
    =&\,2\epsilon C_{u_0}(t+\epsilon),
\end{align*}
and since $r$ and $\epsilon$ were arbitrary, we conclude that $u^{\e_j}(t)$ converges in $L^1_{\loc}(\R)$ to some $u(t)$. Moreover, as $u^{\e_j}(t)$ converges locally to $u(t)$, the bound \eqref{eq: uniformLpBoundUsedInProofOfExistenceForBVData} necessarily carries over to $u(t)$, and so in particular
\begin{align*}
    \norm{u(t)}{L^p(\R)}\leq e^{t\k}\norm{u_0}{L^p(\R)},
\end{align*}
and further by Fatou's lemma we infer for all $t\geq\tilde t\geq0$
\begin{equation}
\begin{split}\label{eq: timeContinuityOfEntropySolution}
    \norm{u(t)-u(\tilde t)}{L^1(\R)}\leq&\, \liminf_{j\to\infty} \norm{u^{\e_j}(t)-u^{\e_j}(\tilde t)}{L^1(\R)}\\ \leq&\,\liminf_{j\to\infty}(t-\tilde t+\e_j)C_{u_0}(t)\\
    =&\,(t-\tilde t)C_{u_0}(t).
\end{split}
\end{equation}
Thus $u\in C([0,\infty), L^1(\R))\cap L^\infty_{\loc}([0,\infty),L^\infty(\R))$.  Next, we prove that $u$ is, in accordance with Def. \ref{def: entropySolution}, an entropy solution of \eqref{eq: theMainEquation} with initial data $u_0$; the latter part follows from $u(0)=u_0$ and \eqref{eq: timeContinuityOfEntropySolution}. To see that $u$ satisfies the entropy inequalities \eqref{eq: entropyInequality}, we pick an arbitrary entropy pair $(\eta,q)$ of \eqref{eq: theMainEquation} and a non-negative $\varphi\in C^\infty_c(\R^+\times\R)$ and recall Lemma \ref{lem: approximateSolutionSatisfiesApproximateEntropySolution} to calculate
\begin{equation}\label{eq: entropyConditionSatisfiedForLimitOfApproximations}
    \begin{split}
        &\,\int_{0}^\infty \int_{\R}\eta(u)\varphi_t +q(u)\varphi_x +\eta'(u)(K\ast u)\varphi\mathrm{d}x\mathrm{d}t\\
    =&\,\lim_{j\to0}\int_{0}^\infty \int_{\R}\eta(u^{\e_j})\varphi_t +q(u^{\e_j})\varphi_x +\eta'(u^{\e_j})(K\ast u^{\e_j})\varphi\mathrm{d}x\mathrm{d}t\\
    \geq&\,  \lim_{j\to0}O(\e_j)=0,
    \end{split}
\end{equation}
where the second line holds as the integrand converges in $L^1(\R)$; after all, $(u^{\e_j})_{j\in\N}$ is uniformly bounded on the compact support of $\varphi$.
By Proposition \ref{prop: uniquenessAndL1Contraction} we conclude that $u$ is the unique entropy solution of \eqref{eq: theMainEquation} with $u_0$ as initial data. What remains to show, is the general limit \eqref{eq: solutionWithBVDataIsLimitOfApproximateSolutions} and the $L^2$ bound of $u$ \eqref{eq: L2BoundOfEntropySolutionWithBVData}; the latter follow by Lemma \ref{lem: theApproximateSolutionMapConservesTheL2Norm} and Fatou's lemma. We prove \eqref{eq: solutionWithBVDataIsLimitOfApproximateSolutions} by contradiction; if this limit does not exist, then there is a subsequence $(u^{\e_j})_{j\in\N}\subset (u^\e)_{\e>0}$, a $t> 0$ and an $r>0$ such that
\begin{align*}
    \liminf_{j\to\infty}\int_{-r}^r|u(t)-u^{\e_j}(t)|\mathrm{d}x>0.
\end{align*}
But as argued above, the infinite set $(u^{\e_j})_{j\in\N}$ must be precompact in $L^1_{\loc}(\R)$ for every $t\geq 0$, and thus we can pick a subsequence converging for every $t\geq 0$ in $L^1_{\loc}(\R)$ to the unique (Proposition \ref{prop: uniquenessAndL1Contraction}) entropy solution $u$ which contradicts the above limit inferior.
\end{proof}

The existence of entropy solutions for general $L^\infty$ data now follows from the previous proposition together with the weighted $L^1$-contraction provided by Proposition \ref{prop: uniquenessAndL1Contraction}. As entropy solutions with $BV$ data are $L^1$-continuous in time, said contraction extends to all $t\geq 0$. 
\begin{corollary}\label{cor: generalSolutionsAreLimitsOfBVSolutions}
For any initial data $u_0\in L^\infty(\R)$, there exists a corresponding entropy solution $u\in C([0,\infty), L^1_{\mathrm{loc}}(\R))$ of \eqref{eq: theMainEquation} satisfying for all $t\geq 0$
\begin{align}\label{eq: coarseSupBoundOfGeneralSolution}
          \norm{u(t)}{L^\infty(\R)}\leq&\, e^{t\k}\norm{u_0}{L^\infty(\R)},
\end{align}
where $\k\coloneqq \norm{K}{L^1(\R)}$. If $u_0\in L^2\cap L^\infty(\R)$, it also satisfies for all $t\geq 0$
\begin{align}
\label{eq: L2BoundOfEntropyGeneralSolution}
    \norm{u(t)}{L^2(\R)}\leq&\,\norm{u_0}{L^2(\R)}.
\end{align}
\end{corollary}
\begin{proof}
For $u_0\in L^\infty(\R)$, let $(u^j)_{j\in\N}$ be a sequence of entropy solutions of \eqref{eq: theMainEquation} whose corresponding initial data $(u_0^j)_{j\in\N}\subset BV(\R)$ satisfies $\sup_j\norm{u_0^j}{L^\infty(\R)}\leq \norm{u_0}{L^\infty(\R)}$ and $u^j_0\to u_0$ in $L^1_{\loc}(\R)$ as $j\to\infty$. For a fixed $T>0$, set
\begin{align*}
    M = e^{T\k}\norm{u_0}{L^\infty(\R)},
\end{align*}
and observe from \eqref{eq: coarseSupBoundOfSolutionWithBVData} that $\sup_j\norm{u^j(t)}{L^\infty(\R)}\leq M$ for all $t\in[0,T]$. Using \eqref{eq: fullL1Contraction}, we find for any $r>0$ 
\begin{align*}
    &\,\limsup_{j,i\to\infty}\sup_{0\leq t\leq T}\int_{-r}^r|u^j(t,x)-u^i(t,x)|\mathrm{d}x\\
    \leq &\,\limsup_{j,i\to\infty}\int_{\R}|u^j_0(x)-u^i_0(x)|w_M^r(T,x)\mathrm{d}x=0,
\end{align*}
where we used that $w_M^r$ is increasing in $t$. This shows that $(u^j)_{j\in\N}$ is Cauchy in the Fréchet space $C([0,\infty), L^1_{\loc}(\R))$ and so the sequence converges to some $u\in C([0,\infty), L^1_\loc(\R))$. Moreover,
\begin{align*}
    \norm{u(t)}{L^\infty(\R)}\leq \liminf_{j\to\infty} \norm{u^j(t)}{L^\infty(\R)}\leq e^{t\k}\norm{u_0}{L^\infty(\R)},
\end{align*}
by \eqref{eq: coarseSupBoundOfSolutionWithBVData}, and so $u\in L^\infty_{\loc}([0,\infty),L^\infty(\R))$ too. That $u$ takes $u_0$ as initial data in $L^1_\loc$-sense follows from the time-continuity of $u$ and $u(0)=\lim_{j\to\infty}u^j_0=u_0$ where the limit is taken in $L^1_\loc(\R)$. Moreover, as each member $(u^j)_{j\in \N}$ satisfies the entropy inequalities \eqref{eq: entropyInequality}, the same can be said for $u$ by a similar calculation as \eqref{eq: entropyConditionSatisfiedForLimitOfApproximations}. Thus the corollary is proved, save for the $L^2$ estimate; this is attained through Fatou's lemma and \eqref{eq: L2BoundOfEntropySolutionWithBVData} as we may assume $\sup_{j}\norm{u_0^j}{L^2(\R)}\leq \norm{u_0}{L^2(\R)}$.
\end{proof}

\subsection{$L^2$ continuity and stability of entropy solutions}\label{sec: L2ContinuityAndStability}
For clarity, we summarize what of Theorem \ref{thm: wellPosedness} has been proved so far and what remains to be proved. Combining  Proposition \ref{prop: uniquenessAndL1Contraction} and Corollary \ref{cor: generalSolutionsAreLimitsOfBVSolutions}, we conclude that there exists a unique entropy solution of \eqref{eq: theMainEquation} in accordance with Def. \ref{def: entropySolution} for every initial data $u_0\in L^\infty(\R)$ and thus also for $u_0\in L^2\cap L^\infty(\R)$. Furthermore, Corollary \ref{cor: generalSolutionsAreLimitsOfBVSolutions} guarantees that these solutions are continuous from $[0,\infty)$ to $L^1_\loc(\R)$ so that the restriction $u(t)\coloneqq u(t,\cdot)\in L^1_{\loc}(\R)$ makes sense for all $t\geq 0$. The same corollary also provides the bounds \eqref{eq: boundsForWellPosednessThm} of Theorem \ref{thm: wellPosedness}. 

It remains to prove that entropy solutions with $L^2\cap L^\infty$ data are continuous from $[0,\infty)$ to $L^2(\R)$ and that they satisfy the stability result of Theorem \ref{thm: wellPosedness}. To do so, we shall exploit the height bound of Corollary \ref{cor: decayingHeightBoundForEntropySolutions}. As explained at the beginning of Section \ref{sec: longTermBounds}, Corollary \ref{cor: decayingHeightBoundForEntropySolutions} can be proved for the case $u_0\in L^2\cap L^\infty(\R)$ independently of this subsection; thus we may here use the height bound \eqref{eq: heightBoundWhenSIsZero} for entropy solutions of \eqref{eq: theMainEquation} without risking a circular argument. From here til the end of the section, we take the above properties of entropy solutions for granted. We begin with a variant of Proposition \ref{prop: uniquenessAndL1Contraction} which makes use of the discussed height bound.

\begin{lemma}\label{lem: L1ContractionExploitingHeigthBound}
There is a function $\Psi\colon [0,\infty)^3\to[0,\infty)$, increasing in all arguments, such that for any pair of entropy solutions $u,v$ of \eqref{eq: theMainEquation} with respective initial data $u_0,v_0\in L^2\cap L^\infty(\R)$ one has for any $t,r\geq 0$ and $N\geq \max\{\norm{u_0}{L^2(\R)},\norm{v_0}{L^2(\R)}\}$ the inequality
\begin{align}\label{eq: L1ContractionExploitingHeigthBound}
    \norm{u(t)-v(t)}{L^1([-r,r])} \leq \Psi(t, N, r)\norm{u_0-v_0}{L^2(\R)}.
\end{align}
\end{lemma}
\begin{proof}
Let $u,v,u_0,v_0$ and $N$ be as described in the lemma. By \eqref{eq: heightBoundWhenSIsZero} from Corollary \ref{cor: decayingHeightBoundForEntropySolutions}, and the property of $N$, we have for all $t>0$
\begin{align}\label{eq: smallMWHichIsAverageOfHeight}
   \frac{\norm{u(t)}{L^\infty(\R)} + \norm{v(t)}{L^\infty(\R)}}{2}\leq C N^{\frac{2}{3}}\Big(1+\frac{1}{t^{\frac{1}{3}}}\Big)\eqqcolon m(t),
\end{align}
where $C\coloneqq \max\{2^{\frac{11}{12}}3^{\frac{1}{3}}\norm{K}{L^1(\R)}^{\frac{1}{3}},2^{\frac{5}{4}}\}$. With $F(u,v)\coloneqq\frac{1}{2}\sgn(u-v)(u^2-v^2)$, we have for any non-negative $\varphi\in C^\infty_c((0,\infty)\times\R)$ the inequality
\begin{align}\label{eq: firstPartOfL2Contraction}
    0\leq \int_0^\infty\int_\R |u-v|\varphi_t + F(u,v)\varphi_x + |u-v|(|K|\ast \varphi)\dd x \dd t.
\end{align}
This is attained by following the first half of the proof of Proposition \ref{prop: uniquenessAndL1Contraction} without using the bound $|F(u,v)|\leq M|u-v|$ as done in the first inequality of \eqref{eq: goingToTheDiagonalFirstPart}; one may instead, when `going to the diagonal', subtract $F(u(t,x),v(t,x))$ from $F(u(t,x),v(s,y))$ and use
\begin{align*}
    |F(u(t,x),v(s,y))-F(u(t,x),v(x,y))|\lesssim |v(s,y)-v(t,x)|,
\end{align*}
which follows from local Lipschitz continuity of $F$ and the fact that $u$ and $v$ are globally bounded (as pointed out after Corollary \ref{cor: decayingHeightBoundForEntropySolutions}). With \eqref{eq: firstPartOfL2Contraction} established, we may \textit{now} filter out $(u+v)/2$ from $F$ using the more precise bound \eqref{eq: smallMWHichIsAverageOfHeight}, that is
\begin{align*}
    |F(u(t,x),v(t,x))|\leq m(t)|u(t,x)-v(t,x)|.
\end{align*}
Doing so, and additionally setting $\varphi(t,x)=\theta(t)\phi(t,x)$ for two arbitrary non-negative functions $\theta\in C^\infty_c((0,T))$ and $\phi\in C_c^\infty((0,T)\times \R)$, with $T>0$ also arbitrary, we conclude from \eqref{eq: firstPartOfL2Contraction} that
\begin{align}\label{eq: secondPartOfL2Contraction}
    0\leq \int_0^T\int_\R|u-v|\theta'\phi\dd x\dd t + \int_0^T\int_\R |u-v|\theta\Big[\phi_t + m(t)|\phi_x| + |K|\ast \phi\Big]\dd x \dd t.
\end{align}
Observe that \eqref{eq: secondPartOfL2Contraction} resembles \eqref{eq: katoInequalityTwo}; for brevity, we skip minor details in the following steps due to their similarity of those following \eqref{eq: katoInequalityTwo}. Let $f\colon\R\to[0,1]$ be a smooth and non-increasing function satisfying $f(x)=1$ for $x\leq 0$ and $f(x)=0$ for sufficiently large $x$, and set
\begin{align*}
    g(t,x)\coloneqq f(|x| + M(t) -M(T)),
\end{align*}
where we have here defined $M(t)$ by
\begin{align*}
    M(t)\coloneqq \int_0^t m(s)\dd s = CN^{\frac{2}{3}}\Big(t + \tfrac{3}{2}t^{\frac{2}{3}}\Big).
\end{align*}
Analogous to \eqref{eq: definitionOfPhi}, we then set
\begin{equation}\label{eq: definitionOfPhiForL2Contraction}
    \begin{split}
           \phi(t,x) =&\, \Big(e^{(T-t)|K|\ast}g(t,\cdot)\Big)(x),
    \end{split}
\end{equation}
and while this $\phi$ is not of compact support, both it, and its derivatives, are integrable on $(0,T)\times \R$ and so by an approximation argument it can be used in \eqref{eq: secondPartOfL2Contraction}. By similar arguments as those following \eqref{eq: definitionOfPhi} we find also here that the second integral in \eqref{eq: secondPartOfL2Contraction} is non-positive, and so we may remove it. Letting then $\theta$ approximate $\mathbbm{1}_{(0,T)}$ in a similar (smooth) manner as done by the sequence \eqref{eq: thetaWhichApproximatesInterval}, we may from \eqref{eq: secondPartOfL2Contraction} conclude
\begin{align}\label{eq: thirdPartOfL2Contraction}
    \int_{\R}|u(T,x)-v(T,x)|\phi(T,x)\dd x\leq \int_{\R}|u_0(x)-v_0(x)|\phi(0,x)\dd x,
\end{align}
where we used that $t\mapsto |u(t,\cdot)-v(t,\cdot)|\phi(t,\cdot)$ is $L^1$-continuous which can be seen by a triangle inequality argument. Note that $\phi(0,x)=f(|x|)$, and so letting $f\to \mathbbm{1}_{(-\infty,r)}$ in $L^1$ sense, the left-hand side of \eqref{eq: thirdPartOfL2Contraction} becomes the left-hand side of \eqref{eq: L1ContractionExploitingHeigthBound}. When $f\to \mathbbm{1}_{(-\infty,r)}$ we also get from \eqref{eq: definitionOfPhiForL2Contraction} that
\begin{align}\label{eq: limitOfPhiForL2Contraction}
    \phi(0,x)\to \Big(e^{T|K|\ast }\mathbbm{1}_{(-\infty,r)}(|\cdot|-M(T))\Big)(x),
\end{align}
in $L^1$ sense. Denoting the right-hand side of \eqref{eq: limitOfPhiForL2Contraction} also by $\phi(0,x)$, it follows by Young's convolution inequality that
\begin{align}\label{eq: L2BoundForPhiInL2Contraction}
    \norm{\phi(0,x)}{L^2(\R)}\leq e^{T\k}[2r+2M(T)]^{\frac{1}{2}} = e^{T\k}\Big[2r+2CN^{\frac{2}{3}}\Big(T + \tfrac{3}{2}T^{\frac{2}{3}}\Big)\Big]^{\frac{1}{2}},
\end{align}
where $\k\coloneqq \norm{K}{L^1(\R)}$. Applying then the Cauchy--Schwarz inequality to the right-hand side of \eqref{eq: thirdPartOfL2Contraction}, and using the above $L^2$ bound for $\phi(0,x)$, we attain \eqref{eq: L1ContractionExploitingHeigthBound} (with $T$ substituting for $t$) for $\Psi(T,N,r)$ given by the right-hand side of \eqref{eq: L2BoundForPhiInL2Contraction}.
\end{proof}
We follow up with a tightness bound for entropy solutions with $L^2\cap L^\infty$ data.
\begin{lemma}\label{lem: tightnessOfL2Solutions}
There is a function $\Phi\colon[0,\infty)^2\times\R\to[0,\infty)$, increasing in all arguments, such that if $u$ is an entropy solution of \eqref{eq: theMainEquation} with initial data $u_0\in L^2\cap L^\infty(\R)$, then for any $t,r\geq 0$ and $N\geq \norm{u_0}{L^2(\R)}$
\begin{align}\label{eq: tightnessOfL2Solutions}
    \int_{|x|>r}u^2(t,x)\dd x \leq \int_\R u_0^2(x)\Phi(t,N,|x|-r)\dd x.
\end{align}
Moreover, 
\begin{align*}
    \lim_{\xi\to-\infty}\Phi(t,N,\xi)=&0, & \Phi(t,N,\xi)=e^{2t\k},\quad \xi>0,
\end{align*}
where $\k\coloneqq \norm{K}{L^1(\R)}$, and in particular, $\xi\mapsto \Phi(t,M,\xi)$ is a bounded function.
\end{lemma}
\begin{proof}
Pick arbitrary initial data $u_0\in L^2\cap L^\infty (\R)$ and let $u$ denote the corresponding entropy solution of \eqref{eq: theMainEquation}. Writing out the entropy inequality \eqref{eq: entropyInequality} for $u$ using the entropy pair $(\eta(u),q(u))=(u^2,\tfrac{2}{3}u^3)$ and a non-negative test function $\varphi\in C^\infty_c((0,T)\times \R)$, with $T\in(0,\infty)$ fixed, we get
\begin{equation}\label{eq: squareEntropyInequality}
    \begin{split}
        0\leq \int_0^T\int_{\R} u^2\varphi_t+\tfrac{2}{3}u^3\varphi_x +2u (K\ast u)\varphi \,\mathrm dx\mathrm dt.
    \end{split}
\end{equation}
By the height bound \eqref{eq: heightBoundWhenSIsZero} of Corollary \ref{cor: decayingHeightBoundForEntropySolutions}, we have $\norm{u(t)}{L^\infty(\R)}\leq m(t)$ where $m(t)$ is as defined in \eqref{eq: smallMWHichIsAverageOfHeight}, and so the second term of the above integrand satisfies 
\begin{align*}
    \tfrac{2}{3}u^3\varphi_x\leq u^2\Big[ \tfrac{2}{3} m(t) |\varphi_x|\Big].
\end{align*}
Additionally, the third part of the integrand satisfies
\begin{align*}
    \int_\R 2u (K\ast u)\varphi \,\mathrm dx =&\, \int_\R\int_\R 2u(t,x)u(t,y) K(x-y) \varphi(t,x) \,\dd y\mathrm dx \\
    \leq &\, \int_\R\int_\R \Big[|u(t,x)|^2 + |u(t,y)|^2\Big] |K(x-y)| \varphi(t,x)\,\dd y\mathrm dx \\
    = &\, \int_\R u^2 \Big[\k \varphi + |K|\ast \varphi\Big]\dd x .
\end{align*}
Inserting these two bounds in \eqref{eq: squareEntropyInequality} we get for any non-negative $\varphi\in C_c^\infty((0,T)\times \R)$
\begin{equation}\label{eq: squareEntropyInequalityTwo}
    \begin{split}
        0\leq \int_0^T\int_{\R} u^2\Big[\varphi_t+\tfrac{2}{3}m(t)|\varphi_x| + \mathcal{K}\ast \varphi\Big] \,\mathrm dx\mathrm dt,
    \end{split}
\end{equation}
where we introduced the measure $\mathcal{K}\coloneqq \k\delta + |K|$, where $\d$ is the Dirac measure. Like in the previous proof, we proceed in a manner similar to the second half of the proof of Proposition \ref{prop: uniquenessAndL1Contraction}, though some necessary changes are made. We set $\varphi(t,x)=\theta(t)\rho(x)\phi(t,x)$ for three smooth non-negative functions on $[0,T]\times \R$ with $\theta$ and $\rho$ having compact support in $(0,T)$ and $\R$ respectively. Additionally, while $\phi$ need not be compactly supported, we require $\phi$ and its derivatives to be bounded. Inserting this in \eqref{eq: squareEntropyInequalityTwo} we get
\begin{equation}\label{eq: squareEntropyInequalityThree}
    \begin{split}
        0\leq \int_0^T\int_{\R} u^2\theta'\rho\phi\dd x\dd t + \int_0^\infty\int_{\R} u^2\theta\Big[\rho\phi_t+\tfrac{2}{3}m(t)|(\rho\phi)_x| + \mathcal{K}\ast (\rho\phi)\Big] \,\mathrm dx\mathrm dt.
    \end{split}
\end{equation}
 Letting $\theta$ approximate $\mathbbm{1}_{(0,T)}$ in a similar (smooth) manner as done by the sequence \eqref{eq: thetaWhichApproximatesInterval}, we may from \eqref{eq: squareEntropyInequalityThree} conclude that 
\begin{equation}\label{eq: squareEntropyInequalityFour}
    \begin{split}
        \int_{\R}u^2(T,x)\rho(x) \phi(T,x)\dd x\leq&\, \int_{\R} u_0^2(x)\rho(x)\phi(0,x)\dd x\\
        &\,+ \int_0^\infty\int_{\R} u^2\Big[\rho\phi_t+\tfrac{2}{3}m(t)|(\rho\phi)_x| + \mathcal{K}\ast (\rho\phi)\Big] \,\mathrm dx\mathrm dt,
    \end{split}
\end{equation}
where we used that $t\mapsto u^2(t,\cdot)\rho(\cdot)\phi(t,\cdot)$ is $L^1$-continuous which can be seen by a triangle inequality argument. Next, we set $\rho(x)=\tilde{\rho}(x/N)$ where $\tilde{\rho}\in C^\infty_c(\R)$ is non-negative and satisfies $\tilde{\rho}(0)=1$. Letting $N\to\infty$, \eqref{eq: squareEntropyInequalityFour} yields by the dominated convergence theorem
\begin{equation}\label{eq: squareEntropyInequalityFive}
    \begin{split}
        \int_{\R}u^2(T,x) \phi(T,x)\dd x\leq&\, \int_{\R} u_0^2(x)\phi(0,x)\dd x\\
        &\,+ \int_0^\infty\int_{\R} u^2\Big[\phi_t+\tfrac{2}{3}m(t)|\phi_x| + \mathcal{K}\ast \phi\Big] \,\mathrm dx\mathrm dt,
    \end{split}
\end{equation}
where the convergence of the integrals follows from the boundness of $\phi$ (and its derivatives) combined with $\norm{u(t)}{L^2(\R)}\leq \norm{u_0}{L^2(\R)}$ for all $t\in[0,T]$. To rid ourselves of the last integral in \eqref{eq: squareEntropyInequalityFive}, we perform a similar trick as done for \eqref{eq: katoInequalityTwo} and \eqref{eq: secondPartOfL2Contraction}, but with a different $f$; we here let $f\colon\R\to[0,1]$ be a \textit{non-decreasing} function with bounded derivatives. Define further $g$ by
\begin{align*}
    g(t,x)\coloneqq f(|x| + M(T) -M(t)),
\end{align*}
where $M(t)$ denotes 
\begin{align}\label{eq: definitionOfLargeMForTightnessBound}
    M(t)\coloneqq \int_0^t \tfrac{2}{3}m(s)\dd s = CN^{\frac{2}{3}}\Big(\tfrac{2}{3}t + t^{\frac{2}{3}}\Big),
\end{align}
and analogues to \eqref{eq: definitionOfPhi}, we set $\phi$ to be
\begin{align*}
    \phi(t,x) = \Big(e^{(T-t)\mathcal{K}\ast}g(t,\cdot)\Big)(x).
\end{align*}
We conclude by similar arguments as those following \eqref{eq: definitionOfPhi} that the square bracket in \eqref{eq: squareEntropyInequalityFive} is non-positive. Thus, removing the non-positive integral in \eqref{eq: squareEntropyInequalityFive} we get
\begin{equation}\label{eq: squareEntropyInequalitySix}
        \int_{\R}u^2(T,x) f(|x|)\dd x\leq \int_{\R} u_0^2(x)\Big(e^{T\mathcal{K}\ast} f(|\cdot| + M(T))\Big)(x)\dd x,
\end{equation}
where we used the explicit expressions for $\phi(T,x)$ and $\phi(0,x)$. Letting $f\to \mathbbm{1}_{(r, \infty)}$ pointwise a.e. it is clear that the left-hand side of \eqref{eq: squareEntropyInequalitySix} converges to $\int_{|x|>r}u^2(T)\dd x$. As for the right-hand side, we get the cumbersome term $e^{T\mathcal{K}\ast}\mathbbm{1}_{(r,\infty)}(|\cdot|+MT)$ which we now simplify. Let the Borel measure $\nu_T$ be defined by the relation $\nu_T\ast = e^{T\mathcal{K}\ast}$ and observe that we for $x\in\R$ have 
\begin{align}\label{eq: simplerExpressionForTailWeight}
    \Big(\nu_T\ast \mathbbm{1}_{(r,\infty)}(|\cdot| + M(T))\Big)(x) =&\, \int_{|x-y|+M(T) > r}\dd\nu_T(y) \leq \int_{|x|-r+M(T) > -|y|}\dd\nu_T(y).
\end{align}
Thus, we define $\Phi(T,N,|x|-r)$ to be the latter expression after substituting for $M(T)$ using \eqref{eq: definitionOfLargeMForTightnessBound}. Inserting this in \eqref{eq: squareEntropyInequalitySix} we get exactly \eqref{eq: tightnessOfL2Solutions} with $T$ substituting for $t$. The properties of $\Phi$ stated in the lemma can be read directly from \eqref{eq: simplerExpressionForTailWeight} when setting $\xi= |x|-r$ together with the fact that $T\mapsto \nu_T$ is increasing (in the canonical sense) and $\int_{\R}d\nu_T = e^{T\mathcal{K}\ast}1=e^{2T\k}$. 
\end{proof}
We may now prove the remaining part of Theorem \ref{thm: wellPosedness}.
\begin{proposition}\label{prop: stabilityOfEntropySolutionsWithL2Data}
 Let two sequences $(t_k)_{k\in\N}\subset [0,\infty)$ and $(u_{0,k})_{k\in\N}\subset L^2\cap L^\infty(\R)$ admit limits
\begin{align*}
    \lim_{k\to\infty}|t_k-t|=&0, & \lim_{k\to\infty}\norm{u_{0,k}-u_0}{L^2(\R)}=0,
\end{align*}
with $t\in [0,\infty)$ and $u_0\in L^2\cap L^\infty(\R)$. Letting $(u_k)_{k\in\N}$ and $u$ denote the entropy solutions of \eqref{eq: theMainEquation} corresponding to the initial data $(u_{0,k})_{k\in\N}$ and $u_0$ respectively, we have
\begin{align*}
    \lim_{k\to\infty}\norm{u_k(t_k)-u(t)}{L^2(\R)}=0.
\end{align*}
In particular, entropy solutions of \eqref{eq: theMainEquation} with $L^2\cap L^\infty$ data are continuous from $[0,\infty)$ to $L^2(\R)$.
\end{proposition}
\begin{proof}
Suppose first that $t>0$. As $t_k\to t$ there is a $T\in(0,\infty)$ such that $(t_k)_{k\in\N}\subset [0,T]$. Similarly, there is an $N$ such that $N\geq \norm{v_0}{L^2(\R)}$ for every $v_0\in \{u_{0,1},u_{0,2},\dots,u_0\}$; observe that such an $N$ satisfies $N\geq \norm{v(t)}{L^2}$ for all $t\in[0,T]$ and $v$ ranging over the corresponding entropy solutions. As the function $\Phi$ from Lemma \ref{lem: tightnessOfL2Solutions} was increasing in its arguments, we infer for all $k\in\N$ and $r>0$ that
\begin{align*}
    \int_{|x|>r}u_k^2(t_k,x)\dd x \leq \int_\R u_{0,k}^2(x)\Phi(T,N,|x|-r).
\end{align*}
Furthermore, as $\xi\mapsto \Phi(T,M,\xi)$ is bounded while $u_{0,k}^2\to u_0^2$ in $L^1(\R)$ as $k\to\infty$, it follows that
\begin{align}\label{eq: uniformTightnessOfL2Sequence}
    \limsup_{k\to\infty}\int_{|x|>r}u_k^2(t_k,x)\dd x\leq \int_\R u_0^2(x)\Phi(T,M,|x|-r),
\end{align}
for any $r>0$. Since $u_0^2$ is integrable and $\lim_{\xi\to-\infty}\Phi(T,M,\xi)=0$, we may for any $\e>0$ pick a sufficiently large $r>0$ such that the right-hand side of \eqref{eq: uniformTightnessOfL2Sequence} is smaller than $\e^2$. For such a couple of constants $\e,r>0$ we find
\begin{align}\label{eq: L2DifferenceBetweenSequenceAndLimit}
    \limsup_{k\to\infty}\norm{u_k(t_k)-u(t)}{L^2(\R)}\leq 2\e + \limsup_{k\to\infty}\norm{u_k(t_k)-u(t)}{L^2([-r,r])}.
\end{align}
To deal with the rightmost term in \eqref{eq: L2DifferenceBetweenSequenceAndLimit}, we yet again let $m$ be the function defined in \eqref{eq: smallMWHichIsAverageOfHeight} using the above $N$. As $t>0$, there are only a finite number of elements in $(t_k)_{k\in\N}$ smaller than $t/2$; without loss of generality, we shall assume there are none. By the height bound \eqref{eq: heightBoundWhenSIsZero} from Corollary \ref{cor: decayingHeightBoundForEntropySolutions} and $m$ being decreasing in $t$, it then follows that $\norm{v}{L^\infty(\R)}\leq m(t/2)$ for every $v\in \{u_{1}(t_1),u_{2}(t_2),\dots,u(t)\}$. Thus, 
\begin{align*}
    \norm{u_k(t_k)-u(t)}{L^2([-r,r])}^2\leq&\, 2m(t/2)\norm{u_k(t_k)-u(t)}{L^1([-r,r])},
\end{align*}
and by the triangle inequality, we further have
\begin{align*}
    \norm{u_k(t_k)-u(t)}{L^1([-r,r])}\leq \norm{u_k(t_k)-u(t_k)}{L^1([-r,r])} +\norm{u(t_k)-u(t)}{L^1([-r,r])}\to0,
\end{align*}
as $k\to0$. Here we used the $L^1_{\loc}$-continuity of $y\mapsto u(t)$ and Lemma \ref{lem: L1ContractionExploitingHeigthBound}. Thus, the last term of \eqref{eq: L2DifferenceBetweenSequenceAndLimit} is zero, and as $\e>0$ was arbitrary, we conclude $\limsup_{k\to\infty}\norm{u_k(t_k)-u(t)}{L^2(\R)}=0$.

Suppose next $t=0$. As above, $u_k(t_k)$ converges to $u(0)=u_0$ in $L^1_{\loc}(\R)$, and in particular, the convergence holds in the sense of distributions. Moreover, we have norm convergence as 
\begin{align*}
  \limsup_{k\to\infty} \norm{u_k(t_k)}{L^2(\R)}\leq \limsup_{k\to\infty} \norm{u_{0,k}}{L^2(\R)}=\norm{u_0}{L^2(\R)},
\end{align*}
while $\norm{u_0}{L^2(\R)}\leq \liminf_{k\to\infty} \norm{u_k(t_k)}{L^2(\R)}$ follows from Fatou's lemma. Thus, we conclude $\norm{u_k(t_k)-u_0}{L^2(\R)}\to 0$ as $k\to\infty$. With the stability result proved, the $L^2$-continuity of $t\mapsto u(t)$ follows by setting $u_{0,k}=u_0$ for all $k\in\N$. 
\end{proof}
We end the section by proving Corollary \ref{cor: extensionToPureL2Data}.
\begin{proof}[Proof of Corollary \ref{cor: extensionToPureL2Data}]
The solution mapping $S$ is by Proposition \ref{prop: stabilityOfEntropySolutionsWithL2Data} jointly continuous from $[0,\infty)\times (L^2\cap L^\infty(\R))^*$ to $L^2(\R)$, where $(L^2\cap L^\infty(\R))^*$ denotes the set $L^2\cap L^\infty(\R)$ equipped with its $L^2$ subspace-topology. Seeking to extend $S$ to all of $[0,\infty)\times L^2(\R)$ in a continuous manner, we note that we have only one choice: whenever a sequence $(u_{0,k})_{k\in \N}\in L^2\cap L^\infty(\R)$ converges in $L^2(\R)$, it follows from Lemma \ref{lem: L1ContractionExploitingHeigthBound} that the corresponding entropy solutions $(u_k)_{k\in\N}$ form a Cauchy sequence in the Fréchet space $C([0,\infty),L^1_{\loc}(\R))$, and thus they converge to a unique element $u\in C([0,\infty),L^1_{\loc}(\R))$ in the appropriate topology. We now argue that $u$ inherits all the nice properties of entropy solutions of \eqref{eq: theMainEquation} established so far, apart from being bounded at $t=0$. Denoting $u_0\in L^2(\R)$ for the $L^2$ limit of $(u_{0,k})_{k\in \N}$, we have by Fatou's lemma
\begin{align*}
    \norm{u(t)}{L^2(\R)}\leq \liminf_{k\to\infty}\norm{u_k(t)}{L^2(\R)}\leq\liminf_{k\to\infty}\norm{u_{0,k}}{L^2(\R)}= \norm{u_0}{L^2(\R)}.
\end{align*}
Moreover, as each $u_{k}$ satisfy the height bound \eqref{eq: heightBoundWhenSIsZero} this bound also carries over to $u$, and thus $u$ is locally bounded in $(0,\infty)\times\R$. Similarly, as each $u_k$ satisfy the entropy inequalities \eqref{eq: entropyInequality}, the same is true for $u$ by a limit argument exploiting the uniform bound of $(u_k)_{k\in\N}$ on the support of $\varphi$ and the fact that $\eta$ and $q$ are smooth; in particular, $u$ is a weak solution of \eqref{eq: theMainEquation}. Even Lemma \ref{lem: L1ContractionExploitingHeigthBound} and Lemma \ref{lem: tightnessOfL2Solutions} carries over to $u$ by approximation. In conclusion, $u$ -- and all other weak solutions obtained this way -- satisfy every property used for entropy solutions in the proof of Proposition \ref{prop: stabilityOfEntropySolutionsWithL2Data}, and so the proposition extends to these weak solutions. Consequently, $S$ is continuous on the larger set $[0,\infty)\times L^2(\R)$, and the proof is complete.
\end{proof}

\section{One-sided Hölder regularity for entropy solutions}\label{sec: longTermBounds}
In this section we show that entropy solutions of \eqref{eq: theMainEquation} with $L^2\cap L^\infty$ data satisfy one-sided Hölder conditions with time-decreasing coefficients. As Subsection \ref{sec: L2ContinuityAndStability} exploits Corollary \ref{cor: decayingHeightBoundForEntropySolutions}, which is proved using the results established here, we stress that the coming analysis will only depend on the results of Subsection \ref{sec: uniquenessOfEntropySolutions} and \ref{sec: existenceOfEntropySolutions}, thus avoiding a circular argument. In Subsection \ref{sec: preliminaryResultsForLongTermBounds} we introduce the necessary building blocks for Subsection \ref{sec: derivingAModulusOfGrowth} where the Hölder conditions are constructed; Theorem \ref{thm: oneSidedHolderRegularity} is proved at the very end of this section. Central in this section is the following object, which in classical terms can be described as a modulus of right upper semi-continuity.
\begin{definition}\label{def: modulusOfGrowth}
We say that a smooth and strictly increasing function $\omega\colon(0,\infty)\to(0,\infty)$ is a \textit{modulus of growth} for $v\colon\R\to\R$ if for all $h>0$
\begin{align*}
    \esssup_{x\in\R}\Big[v(x+h)-v(x)\Big]\leq \omega(h).
\end{align*}
\end{definition}
The requirement that $\omega$ be smooth and strictly increasing is for technical convenience. Note also that we did not require $\omega(0+)=0$; this is to include the expression \eqref{eq: generalEvolutionOfOleinikEstimateUnderConvolutionSolutionMap} when $s=0$.

\subsection{Preliminary results}\label{sec: preliminaryResultsForLongTermBounds}
The classical Ole\v{ı}nik estimate \cite{MR2574377} for entropy solutions of Burgers' equation is for $(t,x)\in \R^+\times \R$ and $h\geq0$ given by
\begin{align}\label{eq: classicalOleinikEstimateForBurgersSolutions}
    u(t,x+h)-u(t,x)\leq \frac{h}{t}.
\end{align}
For a fixed $t>0$, this one-sided Lipschitz condition (or modulus of growth) restricts how fast $x\mapsto u(t,x)$ can grow, but not how fast it can decrease, thus allowing for jump discontinuities (shocks) whose left limit is above the right. Interestingly, when the initial data of Burgers' equation satisfies $u_0\in L^p(\R)$ for some $p\in [1,\infty)$, one can for the corresponding entropy solution $u$ use \eqref{eq: classicalOleinikEstimateForBurgersSolutions} to attain
\begin{align}\label{eq: decayOfSolutionsOfBurgersEquation}
    \norm{u(t)}{L^\infty(\R)}^{p+1}\leq\tfrac{p+1}{t}\norm{u(t)}{L^p(\R)}^p\leq \tfrac{p+1}{t}\norm{u_0}{L^p(\R)}^p,
\end{align}
where the rightmost inequality is just the classical $L^p$ bound for Burgers' equation, and thus, the height of $u(t)=u(t,\cdot)$ tends to zero as $t\to\infty$. We omit the proof of \eqref{eq: decayOfSolutionsOfBurgersEquation}, which is similar to that of the next lemma where we provide a general method for bounding the height of a function $u\in L^2(\R)$ admitting a modulus of growth $\omega$. We focus on $L^2(\R)$ because other $L^p$ norms might fail to be non-increasing for entropy solutions of \eqref{eq: theMainEquation}; the generalization of  \eqref{eq: classicalOleinikEstimateForBurgersSolutions} will require a generalization of \eqref{eq: decayOfSolutionsOfBurgersEquation}, so $p=2$ is the natural choice as $\norm{u(t)}{L^2(\R)}\leq \norm{u_0}{L^2(\R)}$ for entropy solutions of \eqref{eq: theMainEquation}. In the coming lemma we also provide for later convenience a bound on the following seminorm defined for $v\in L^\infty(\R)$ by
\begin{align}\label{eq: definitionOfSeminormHeight}
    |v|_\infty \coloneqq \esssup_{x,y\in \R}\frac{v(x) - v(y)}{2}.
\end{align}
As $|v|_{\infty}\leq \norm{v}{L^\infty(\R)}$, any bound on $\norm{v}{L^\infty(\R)}$ obviously carries over to $|v|_{\infty}$. Note however, that the next lemma bounds $|v|_{\infty}$ sharper than it does $\norm{v}{L^\infty(\R)}$.
Finally, we mention that the extra assumptions posed on $\omega$ in the lemma are only for technical simplicity, as the lemma holds more generally.

\begin{lemma}\label{lem: theBoundOnSupFromL2Norm}
Let $v\in L^2 (\R)$ admit a modulus of growth $\omega$ that satisfies $\omega(0+)=0$ and $\omega(\infty)=\infty$. Then $v\in L^2\cap L^\infty(\R)$ and moreover
\begin{align}\label{eq: theBoundOnSupFromL2Norm}
    \norm{v}{L^2(\R)}^2\geq&\, F\Big(\norm{v}{L^\infty(\R)}\Big),\\ \label{eq: theBoundOnDoubleSupFromL2Norm}
      \tfrac{1}{2}\norm{v}{L^2(\R)}^2\geq&\, F\Big(|v|_{\infty}\Big),
\end{align}
where $F$ is the strictly increasing and convex function
\begin{align}\label{eq: definitionOfFOfInverseOmega}
    F(y)\coloneqq 2\int_0^{y}\int_0^{y_1} \omega^{-1}(y_2)\mathrm{d}y_2\mathrm{d}y_1.
\end{align}
\end{lemma}
\begin{proof}
By Lemma \ref{lem: modulusOfGrowthImpliesLeftAndRightLimits} from the appendix we may assume $v$ to be left-continuous, and in particular, well defined at every point. Then, for all $x\in\R$ such that $v(x)\geq0$ we have for $h\in (0,\omega^{-1}(v(x))]$
\begin{align*}
v(x-h)\geq v(x)-\omega(h)\geq 0,
\end{align*}
and similarly, for all $x\in \R$ such that $v(x)<0$ we have for $h\in (0,\omega^{-1}(-v(x))]$
\begin{align*}
   v(x+h)\leq v(x)+\omega(h)\leq 0.
\end{align*}
Squaring each of these inequalities (the bottom one would flip direction) and integrating over $h\in(0,\omega^{-1}(|v(x)|)]$, yields in both cases
\begin{align}\label{eq: quarterWayAtProvingBoundOnSupFromL2Norm}
    \norm{v}{L^2(\R)}^2\geq \int_0^{\omega^{-1}(|v(x)|)} (|v(x)|-\omega(h))^2 \mathrm{d}h,
\end{align}
where the left-hand side has been replaced by the upper bound $\norm{v}{L^2(\R)}^2$. Performing the change of variables $h=\omega^{-1}(y)$ the right-hand side of \eqref{eq: quarterWayAtProvingBoundOnSupFromL2Norm} can further be written
\begin{align*}
    \int_0^{|v(x)|} (|v(x)|-y)^2 \mathrm{d}\omega^{-1}(y) =&\,  2\int_0^{|v(x)|} (|v(x)|-y) \omega^{-1}(y)\mathrm{d}y\\
    =&\,  2\int_0^{|v(x)|}\int_0^y \omega^{-1}(z)\mathrm{d}z\mathrm{d}y,
\end{align*}
where we integrated by parts twice. This last expression is exactly $F(|v(x)|)$, and so letting this replace the right-hand side of \eqref{eq: quarterWayAtProvingBoundOnSupFromL2Norm} followed by taking the supremum with respect to $x\in\R$ yields \eqref{eq: theBoundOnSupFromL2Norm}. For \eqref{eq: theBoundOnDoubleSupFromL2Norm}, we write $v_+$ and $v_-$ for the positive and negative part of $v$ respectively, and observe that $v\in L^2\cap L^\infty(\R)$ implies $|v|_\infty=\frac{1}{2}(\norm{v_+}{L^\infty(\R)} + \norm{v_-}{L^\infty(\R)})$
and $\norm{v}{L^2(\R)}^2 = \norm{v_+}{L^2(\R)}^2+\norm{v_-}{L^2(\R)}^2$. Furthermore, as both $v_+$ and $-v_-$ admit $\omega$ as a modulus of growth, we can use \eqref{eq: theBoundOnSupFromL2Norm} followed by Jensen's inequality to calculate
\begin{align*}
   \tfrac{1}{2}\norm{v}{L^2(\R)}^2=&\, \tfrac{1}{2}\Big[\norm{v_+}{L^2(\R)}^2 + \norm{v_-}{L^2(\R)}^2\Big]\\
   \geq &\,\tfrac{1}{2}\Big[F\Big(\norm{v_+}{L^\infty(\R)}\Big) + F\Big(\norm{v_-}{L^\infty(\R)}\Big)\Big]\\
   \geq&\,  F\Big(\tfrac{1}{2}\Big[\norm{v_+}{L^\infty(\R)} + \norm{v_-}{L^\infty(\R)}\Big]\Big)\\
   =&\,F\Big(|v|_{\infty}\Big).
\end{align*}
\end{proof}

 The calculations of the next subsection, where Theorem \ref{thm: oneSidedHolderRegularity} is proved, can be boiled down to the three lemmas of this subsection (Lemma \ref{lem: theBoundOnSupFromL2Norm} being the first). The remaining Lemma \ref{lem: generalEvolutionOfOleinikEstimateUnderBurgersSolutionMap} and Lemma \ref{lem: generalEvolutionOfOleinikEstimateUnderConvolutionSolutionMap}, induce a natural evolution of a modulus of growth from the mappings $S^B_t$ and $S^K_t$, introduced in \eqref{eq: definitionOfSolutionMapOfBurgersEquation} and \eqref{eq: definitionOfSolutionMapOfConvolutionEquation}. The relevance of these results should come as no surprise; the previous section showed that entropy solutions could be approximated by repeated compositions of said mappings.

\begin{lemma}\label{lem: generalEvolutionOfOleinikEstimateUnderBurgersSolutionMap}
Suppose $v\in BV(\R)$ admits a concave modulus of growth $\omega$. Then for any $\e>0$, the function $w=S^B_\e( v)$, admits the modulus of growth
\begin{align}\label{eq: generalEvolutionOfOleinikEstimateUnderBurgersSolutionMap}
    h\mapsto \frac{\omega(h)}{1+\e\omega'(h)}.
\end{align}
\end{lemma}
\begin{proof}
As $S^B_t$ maps $BV$ to itself, both $v$ and $w$ admits essential left and right limits at point. Thus, we assume without loss of generality that they are left continuous. For $x\in\R$, $h> 0$ and $t\in[0,\e]$, introduce the two (minimal) backward characteristics of $S^B_t(v)$ emanating from $(\e,x)$ and $(\e,x+h)$ respectively
\begin{align*}
    \xi_1(t)&\,= x + (t-\e)w(x),\\
     \xi_2(t)&\,= x+h + (t-\e)w(x+h).
\end{align*}
As $v$ and $w$ are left continuous, it follows from  Theorem 11.1.3. in \cite{MR2574377} that
\begin{align*}
    v(\xi_1(0))\leq &\, w(x), &
    w(x+h)\leq &\, v(\xi_2(0)+).
\end{align*}
Moreover, by the Ole\v{ı}nik estimate of $w$  \eqref{eq: classicalOleinikEstimateForBurgersSolutions}, we find
\begin{align*}
    \xi_2(0)-\xi_1(0)=&\,h-\e[w(x+h)-w(x)]\geq 0,
\end{align*}
and so exploiting $\omega$ we can calculate
\begin{equation}\label{eq: tracingTheCharacteristicsBackwards}
    \begin{split}
        w(x+h)-w(x)
    &\,\leq v( \xi_2(0)+)-v( \xi_1(0))\\
    &\,\leq \omega(h-\e[w(x+h)-w(x)])\\
    &\,\leq \omega(h) -\e\omega'(h)(w(x+h)-w(x)),
    \end{split}
\end{equation}
where the last inequality holds as $\omega$ is concave. We conclude that
\begin{align*}
    w(x+h)-w(x)\leq \frac{\omega(h)}{1+\e\omega'(h)},
\end{align*}
for all $x\in\R$ and $h> 0$. That \eqref{eq: generalEvolutionOfOleinikEstimateUnderBurgersSolutionMap} is positive, smooth and strictly increasing follows from $\omega$ being positive, smooth, strictly increasing and concave.
\end{proof}

We follow immediately with a similar result for the operator $S^K_t$, which will depend on the fractional variation $|K|_{TV^s}$ as defined in \eqref{eq: definitionOfFractionalVariance} and the seminorm  $|\cdot|_{\infty}$ defined in \eqref{eq: definitionOfSeminormHeight}.

\begin{lemma}\label{lem: generalEvolutionOfOleinikEstimateUnderConvolutionSolutionMap}
Let $s\in[0,1]$ and assume $\snorm{K}{TV^s}<\infty$. Suppose $v\in L^\infty(\R)$ admits a modulus of growth $\omega$. Then for any $\e>0$, the function $w=S^K_\e (v)$ admits the modulus of growth
\begin{align}\label{eq: generalEvolutionOfOleinikEstimateUnderConvolutionSolutionMap}
    h\mapsto  \omega(h) + \e\snorm{K}{TV^s}|v|_{\infty}h^s.
\end{align}
\end{lemma}
\begin{proof}
For simple notation, we introduce the shift operator $T_h\colon f\mapsto f(\cdot + h)$. As shifts commute with convolution, and since $\int_{\R}T_h K - K\mathrm{d}x = 0$, we start by noting that for any $k\in\R$
\begin{align*}
    (T_h-1) (K\ast v)&\,=[(T_h-1)K] \ast (v-k).
\end{align*}
Next, we introduce $\overline v=\esssup_x v(x)$ and $\underline v=\essinf_x v(x)$, and we observe that
\begin{align*}
    \norm{v-k}{L^\infty(\R)}=\max\{\overline v -k,k-\underline v\}.
\end{align*}
Thus, we minimize by setting $k=\frac{1}{2}(\overline v + \underline v)$ and get $ \norm{v-k}{L^\infty(\R)}= \frac{1}{2}(\overline{v}-\underline{v}) = |v|_{\infty}$.
By Young's convolution inequality and the above calculations we infer
\begin{align*}
        \norm{(T_h-1) (K\ast v)}{L^\infty(\R)}&\,\leq\norm{K(\cdot + h)-K}{L^1(\R)}\norm{v-k}{L^\infty(\R)}\\
    &\,\leq \snorm{K}{TV^s}|v|_{\infty} h^s.
\end{align*}
For any $h> 0$ we then conclude
\begin{align*}
    (T_h-1)w =&\, (T_h-1)v + \e(T_h-1)(K\ast v)
    \leq
    \omega(h) +  \e\snorm{K}{TV^s}|v|_\infty h^s,
\end{align*}
where the last inequality holds pointwise a.e. in $\R$.
\end{proof}

\subsection{Deriving a modulus of growth for entropy solutions.}\label{sec: derivingAModulusOfGrowth}
Throughout this subsection we consider $s\in[0,1]$ fixed and assume that $|K|_{TV^s}$ is finite. Further, we let $\mu,\k_s\in (0,\infty)$ denote arbitrary fixed values, though we impose the requirement $\k_s\geq |K|_{TV^s}$. The role of $\mu$ and $\k_s$ will essentially be that of placeholders for the $L^2$ norm of the initial data and of $|K|_{TV^s}$ respectively, but note that $\mu$ and $\k_s$ are strictly positive (even if the quantities they represent might be zero). This positivity is for technical convenience as some of the coming expressions would otherwise need a limit sense interpretation.

We shall for an arbitrary entropy solution $u$ of \eqref{eq: theMainEquation} with $L^2\cap L^\infty$ data, seek an expression $a(t)$ such that $h\mapsto a(t)h^{\frac{1+s}{2}}$ serves as a modulus of growth (Def. \ref{def: modulusOfGrowth}) for $x\mapsto u(t,x)$.  We begin with an important result, which among other things rephrases Lemma \ref{lem: theBoundOnSupFromL2Norm} for the more explicit case $\omega(h)=ah^{\frac{1+s}{2}}$. For this purpose, we introduce the constant
\begin{align}\label{eq: definitionOfCs}
 c_s=&\,\Bigg[\frac{(2+s)(3+s)}{2(1+s)^2}\Bigg]^{\frac{1+s}{4+2s}},
\end{align}
and the function
\begin{align}\label{eq: definitionOfHOfA}
    H(a)=&\,\frac{(2c_s)^{\frac{2}{1+s}}\mu^{\frac{2}{2+s}}}{a^{\frac{2}{2+s}}},
\end{align}
defined for all $a>0$. We also recall definition \eqref{eq: definitionOfSeminormHeight} of the seminorm $|\cdot|_{\infty}$. The essential part of the next lemma is in allowing us to extend the domain for which a homogeneous modulus of growth is valid. This will be vital when proving the following proposition.

\begin{lemma}\label{lem: explicitBoundOnSupFromExplicitOmegaFunction}
With fixed $a>0$, define $\omega(h)=ah^{\frac{1+s}{2}}$. Suppose $v\in L^2(\R)$ satisfies $\norm{v}{L^2(\R)}\leq \mu$ and admits $\omega$ as a modulus of growth for $h\in (0,H(a))$.
Then $v$ admits $\omega$ as a modulus of growth for all $h\in(0,\infty)$ and moreover
\begin{align}\label{eq: explicitBoundOnSupFromExplicitOmegaFunction}
    \norm{v}{L^\infty(\R)}\leq&\, 2^{\frac{1+s}{4+2s}}c_s \mu^{\frac{1+s}{2+s}}a^{\frac{1}{2+s}},\\\label{eq: explicitBoundOnDoubleSupFromExplicitOmegaFunction}
    |v|_\infty\leq &\, c_s \mu^{\frac{1+s}{2+s}}a^{\frac{1}{2+s}}.
\end{align}
\end{lemma}
\begin{proof}
We begin by proving the two inequalities, so let us assume for now that $v$ admits $\omega$ as a modulus of growth for all $h\in(0,\infty)$. Since $\omega^{-1}(y)=a^{-\frac{2}{1+s}}y^{\frac{2}{1+s}}$ the function $F$ from \eqref{eq: definitionOfFOfInverseOmega} can here be written
\begin{align*}
    F(y)=\Bigg[\frac{2(1+s)^2}{(3+s)(4+2s)}\Bigg]\frac{y^{\frac{4+2s}{1+s}}}{a^{\frac{2}{1+s}}} = \frac{1}{2}\Bigg(\frac{y}{c_s a^{\frac{1}{2+s}}}\Bigg)^{\frac{4+2s}{1+s}},
    \end{align*}
with inverse
\begin{align*}
    F^{-1}(y)=2^{\frac{1+s}{4+2s}}c_sa^{\frac{1}{2+s}}y^{\frac{1+s}{4+2s}}.
    \end{align*}
Combined with $\norm{v}{L^2(\R)}\leq \mu$, \eqref{eq: theBoundOnSupFromL2Norm} and \eqref{eq: theBoundOnDoubleSupFromL2Norm} give $\norm{v}{L^\infty(\R)}\leq F^{-1}(\mu^2)$ and $|v|_{\infty}\leq F^{-1}(\frac{1}{2}\mu^2)$, which coincides with \eqref{eq: explicitBoundOnSupFromExplicitOmegaFunction} and \eqref{eq: explicitBoundOnDoubleSupFromExplicitOmegaFunction} respectively. Next, assume we only know that $v$ admits $\omega$ as a modulus of growth for $h\in(0,H(a))$. The steps in the proof of Lemma \ref{lem: theBoundOnSupFromL2Norm} can still be carried out if one lets the role of $\omega^{-1}(y)=a^{-\frac{2}{1+s}}y^{\frac{2}{1+s}}$ be taken by the truncated version
\begin{align*}
    y\mapsto \min\Big\{a^{-\frac{2}{1+s}}y^{\frac{2}{1+s}},H(a)\Big\},
\end{align*}
to yield the inequalities $\norm{v}{L^\infty(\R)}\leq \tilde{F}^{-1}(\mu^2)$ and $|v|_{\infty}\leq \tilde{F}^{-1}(\frac{1}{2}\mu^2)$ with
\begin{align*}
     \tilde{F}(y)\coloneqq 2\int_0^{y}\int_0^{y_1} \min\Big\{a^{-\frac{2}{1+s}}y_2^{\frac{2}{1+s}},H(a)\Big\}\mathrm{d}y_2\mathrm{d}y_1.
\end{align*}
As $\tilde{F}$ is strictly increasing and agrees with $F$ on $(0,aH(a)^{\frac{1+s}{2}})$, we necessarily have both $\tilde{F}^{-1}(\mu^2)=F^{-1}(\mu^2)$ and $\tilde{F}^{-1}(\frac{1}{2}\mu^2)=F^{-1}(\frac{1}{2}\mu^2)$ provided $F^{-1}(\mu^2)<aH(a)^{\frac{1+s}{2}}$. As $F^{-1}(\mu^2)$ is exactly the right-hand side of \eqref{eq: explicitBoundOnSupFromExplicitOmegaFunction}, we see that the latter inequality holds since
\begin{align*}
    F^{-1}(\mu^2)= 2^{\frac{1+s}{4+2s}}c_s\mu^{\frac{1+s}{2+s}}a^{\frac{1}{2+s}} < 2c_s\mu^{\frac{1+s}{2+s}}a^{\frac{1}{2+s}}=aH(a)^{\frac{1+s}{2}}.
\end{align*}
Thus, the bounds for $\norm{v}{L^\infty(\R)}$ and $|v|_{\infty}$ attained now coincides again with \eqref{eq: explicitBoundOnSupFromExplicitOmegaFunction} and \eqref{eq: explicitBoundOnDoubleSupFromExplicitOmegaFunction}. It then follows that $v$ admits $\omega$ as a modulus of growth for all $h\in(0,\infty)$; indeed, for any $h\in[H(a),\infty)$ we have the two trivial inequalities
\begin{align*}
    \esssup_{x\in\R} \Big[v(x+h)-v(x)\Big]\leq&\, 2|v|_\infty, & aH(a)^{\frac{1+s}{2}}\leq ah^{\frac{1+s}{2}},
\end{align*}
and so we would be done if $2|v|_\infty\leq aH(a)^{\frac{1+s}{2}}$, which is precisely the already established inequality \eqref{eq: explicitBoundOnDoubleSupFromExplicitOmegaFunction} multiplied by two.
\end{proof}
The next proposition combines Lemma \ref{lem: generalEvolutionOfOleinikEstimateUnderBurgersSolutionMap} and \ref{lem: generalEvolutionOfOleinikEstimateUnderConvolutionSolutionMap} to attain a corresponding result for the operator $S^B_\e\circ S^K_\e$. While it in Section \ref{sec: wellPosedness} was natural to work with iterations of $S^K_\e\circ S^B_\e$, it will here be easier to work with its counterpart $S^B_\e\circ S^K_\e$. We now introduce the useful limit value $\underline{a}$ defined by
\begin{align}\label{eq: limitValueA}
    \underline{a} =&\,\bigg(\frac{2c_s\kappa_s}{1+s}\bigg)^{\frac{2+s}{3+2s}}\mu^{\frac{1+s}{3+2s}}.
\end{align}
This quantity will naturally occur in our calculations to come; it relates to the sought coefficient $a(t)$ through the relation $\lim_{t\to\infty}a(t)=\underline{a}$.

\begin{proposition}\label{prop: newHolderCoefficientAfterBothSolutionMapsHaveBeenAdded}
For every $A>\underline{a}$, there are constants $C_A, \e_A>0$ such that: if $v\in BV(\R)$ satisfies $\norm{v}{L^2(\R)}\leq \mu$ and admits the modulus of growth $h\mapsto ah^{\frac{1+s}{2}}$ for some $a\in[\underline{a}, A]$, then for every $\e\in(0,\e_A]$ the function $w=S^B_\e\circ S^K_\e(v)$ admits the modulus of growth 
\begin{align}
    h\mapsto \Big(a-\e f(a) + \e^2 C_A\Big)h^{\frac{1+s}{2}},
\end{align}
where $f(a)\geq 0$ is given by
\begin{align}\label{eq: definitionOfFunctionF}
          f(a)= \Bigg[\frac{(1+s) a^{\frac{2-s}{2+s}}}{2^{\frac{2}{1+s}}c_s^{\frac{1-s}{1+s}}\mu^{\frac{1-s}{2+s}}}\Bigg]\bigg[a^{\frac{3+2s}{2+s}}-\underline{a}^{\frac{3+2s}{2+s}}\bigg] .
\end{align}
\end{proposition}
\begin{proof}
 For fixed $A>\underline{a}$, let $v\in BV(\R)$ and $a\in[\underline{a},A]$ be as described in the lemma. We fix the pair $v$ and $a$ for convenience, but it should be clear from the proof that the construction of $C_A$ and $\e_A$ do not in fact depend on said pair. Introduce for $\e>0$ the auxiliary function $\tilde{v}= S^K_\e(v)$. Combining Lemma \ref{lem: generalEvolutionOfOleinikEstimateUnderConvolutionSolutionMap} and \eqref{eq: explicitBoundOnDoubleSupFromExplicitOmegaFunction}, $\tilde{v}$ admits the concave modulus of growth
\begin{align*}
    \tilde\omega(h)=ah^{\frac{1+s}{2}} + \e c_s\kappa_sa^{\frac{1}{2+s}}\mu^{\frac{1+s}{2+s}}h^s,
\end{align*}
where $|K|_{TV^s}$ was replaced by the larger $\k_s$ introduced at the beginning of this subsection.
And since $\tilde{v}\in BV(\R)$, as follows from \eqref{eq: propertyOneOfBurgerMap} and \eqref{eq: propertyThreeOfBurgerMap}, we can further apply Lemma \ref{lem: generalEvolutionOfOleinikEstimateUnderBurgersSolutionMap} to $w= S^B_\e(\tilde{v})$, which combined with $ \tilde\omega'(h)> (\tfrac{1+s}{2})ah^{\frac{s-1}{2}}$, allows us to conclude that $w$ admits the modulus of growth
\begin{equation}\label{eq: aRepresentationOfTildeOmegaV}
    \begin{split}
            \omega(h)=&\, \frac{a h^{\frac{1+s}{2}} + \e c_s\kappa_s a^{\frac{1}{2+s}}\mu^{\frac{1+s}{2+s}}h^s}{1+\e (\tfrac{1+s}{2})a h^{\frac{s-1}{2}}}\\
    =&\, ah^{\frac{1+s}{2}} + \frac{-\e (\frac{1+s}{2}) a^2 h^{s} + \e c_s\kappa_s a^{\frac{1}{2+s}}\mu^{\frac{1+s}{2+s}}h^s}{1+\e ({\tfrac{1+s}{2}})a h^{\frac{s-1}{2}}}\\
    =&\, ah^{\frac{1+s}{2}} - \e \underbrace{\Bigg[\frac{ (1+s)a^2-2c_s\kappa_sa^{\frac{1}{2+s}}\mu^{\frac{1+s}{2+s}}}{2h^{\frac{1-s}{2}}+\e (1+s)a}\Bigg]}_{B(a,h,\e)} h^{\frac{1+s}{2}},
    \end{split}
\end{equation}
where $B(a,h,\e)$ denotes the square bracket. With $\underline{a}$ as given by \eqref{eq: limitValueA}, this square bracket can further be factored
\begin{align}\label{eq: definitionOfB}
    B(a,h,\e)=\Bigg[\frac{(1+s)a^{\frac{1}{2+s}}}{2h^{\frac{1-s}{2}}+\e(1+s)a}\Bigg]\bigg[a^{\frac{3+2s}{2+s}}-\underline{a}^{\frac{3+2s}{2+s}}\bigg].
\end{align}
Since $a\geq\underline{a}$ it follows that $B(a,h,\e)$ is non-negative and thus non-increasing in $h>0$. Consequently, we read from \eqref{eq: aRepresentationOfTildeOmegaV} the inequality
\begin{align}\label{eq: modulusOfGrowthBoundedByHomogeneousOneOnBoundedIntervals}
    \omega(h)\leq&\, \Big(a-\e B(a,\overline{h},\e)\Big)h^{\frac{1+s}{2}}, & 0<h<&\, \overline{h}.
\end{align}
Since \eqref{eq: modulusOfGrowthBoundedByHomogeneousOneOnBoundedIntervals} can be viewed as implying that $w$ admits a homogeneous modulus of growth on bounded intervals, we would like to make use of Lemma \ref{lem: explicitBoundOnSupFromExplicitOmegaFunction}; however, we do not necessarily have $\norm{w}{L^2(\R)}\leq \mu$ (as is assumed by said lemma). We deal with this small inconvenience as follows: define $\tilde w$ by
\begin{align}\label{eq: renormalizedV}
    \tilde{w} \coloneqq &\,\rho^{-1} w, & \rho\coloneqq \, \max\Big\{1,\mu^{-1}\norm{w}{L^2(\R)}\Big\},
\end{align}
that is, $\tilde{w}$ is the renormalized version of $w$ if the $L^2$ norm of $w$ exceeds $\mu$. We proceed by proving the proposition for $\tilde w$ and then extend the result to $w$. Observe that $\omega$ must serve as a modulus of growth also for $\tilde w$ since $\rho\geq 1$, and consequently by \eqref{eq: modulusOfGrowthBoundedByHomogeneousOneOnBoundedIntervals}, $\tilde w$ further admits for any fixed $\overline{h}>0$ the modulus of growth
\begin{align}\label{eq: modulusOfGrowthOnFixedIntervals}
    h\mapsto&\,\Big(a-\e B(a,\overline{h},\e)\Big)h^{\frac{1+s}{2}},
\end{align}
for the restricted values $h\in (0,\overline{h})$. Lemma \ref{lem: explicitBoundOnSupFromExplicitOmegaFunction} then tells us that $\tilde w$ must additionally admit \eqref{eq: modulusOfGrowthOnFixedIntervals} as a modulus of growth for all $h>0$ provided 
\begin{align}\label{eq: necessaryInequalityForExtensionOfModulusOfGrowth}
    H\Big(a-\e B(a,\overline{h},\e)\Big)\leq \overline{h},
\end{align}
where the function $H$ is as defined by \eqref{eq: definitionOfHOfA}. We now show that there is an appropriate constant $D_A$ so that $\overline{h}=H(a)+\e D_A$ satisfies \eqref{eq: necessaryInequalityForExtensionOfModulusOfGrowth}. To do so, we start by introducing the closed set of points $(a,h,\e)$ defined by
\begin{align*}
    S_A = [\underline{a},A]\times[H(A),\infty)\times[0,\infty),
\end{align*}
where we abuse notation slightly by reusing $a$ as a dummy variable for referring to elements in $[\underline{a},A]$ (although the original $a\in [\underline{a},A]$ is fixed).
From \eqref{eq: definitionOfB} we see that both $(a,h,\e)\mapsto B(a,h,\e)$ and its partial derivatives are bounded on the set $S_A$. We exploit the additional smoothness of $B$ later; for now we need only $\norm{B}{L^\infty(S_A)}<\infty$. Pick $\e_A>0$ such that 
\begin{align*}
    \e_A\norm{B}{L^\infty(S_A)}\leq \tfrac{1}{2}\underline{a},
\end{align*}
and observe that the argument of $H$ in \eqref{eq: necessaryInequalityForExtensionOfModulusOfGrowth} must then lie in  $[\frac{1}{2}\underline{a},A]$ for all $(a,\overline{h},\e)\in [\underline{a},A]\times [H(a),\infty)\times [0,\e_A]\subset S_A$. Moreover, as $H$ is smooth on $[\frac{1}{2}\underline{a},A]$ we conclude for any such triplet $(a,\overline{h},\e)$ that 
\begin{align*}
     H\Big(a-\e B(a,\overline{h},\e)\Big)\leq H(a) + \e\norm{H'}{L^\infty([\frac{1}{2}\underline{a}, A])}\norm{B}{L^\infty(S_A)}\eqqcolon H(a)+\e D_A.
\end{align*}
Thus, the choice $\overline{h}\coloneqq H(a)+\e D_A$ satisfies \eqref{eq: necessaryInequalityForExtensionOfModulusOfGrowth} for every $a\in[\underline{a},A]$ and $\e\in(0,\e_A]$, and so substituting for $\overline{h}$ in \eqref{eq: modulusOfGrowthOnFixedIntervals}, we conclude that $\tilde w$ admits the modulus of growth
\begin{align}\label{eq: explicitNewModulusOfGrowthForVTilde}
    h\mapsto \Big(a-\e B(a,H(a) + \e D_A,\e)\Big)h^{\frac{1+s}{2}},
\end{align}
for all $h>0$, provided $\e\in(0,\e_A]$ and $a\in[\underline{a},A]$ (the latter already assumed). Recalling that the partial derivatives of $B$ are bounded on $S_A$, we can write
\begin{align}\label{eq: exploitingThatThePartialDerivativesOfBAreBoundedOnSA}
    B(a,H(a)+\e D_A,\e)\geq  B(a, H(a),0) - \e \Big[ D_A\norm{\tfrac{\partial B}{\partial h}}{L^\infty(S_A)} + \norm{\tfrac{\partial B}{\partial \e}}{L^\infty(S_A)}\Big],
\end{align}
and so letting $C_A$ denote a constant no smaller than the square bracket in \eqref{eq: exploitingThatThePartialDerivativesOfBAreBoundedOnSA}, we combine this inequality with \eqref{eq: explicitNewModulusOfGrowthForVTilde} to further conclude that
\begin{align}\label{eq: explicitsecondNewModulusOfGrowthForVTilde}
    h\mapsto \Big(a-\e B(a,H(a),0)+ \e^2 C_A\Big)h^{\frac{1+s}{2}},
\end{align}
also serves as a modulus of growth for $\tilde w$, again with $\e\in(0,\e_A]$ and $a\in[\underline{a},A]$. Using the explicit expressions \eqref{eq: definitionOfB} and \eqref{eq: definitionOfHOfA} one attains the identity $B(a,H(a),0)=f(a)$, where $f$ is defined in \eqref{eq: definitionOfFunctionF}, and so the proposition has been proved for the renormalized function $\tilde w$. It remains to extend the result to $w$; assume from here on out that $\e\in(0,\e_A]$. Introducing $\tilde a=(a-\e f(a) + \e^2 C_A)$ for brevity, it is clear from the relation $w = \rho \tilde w$, where $\rho$ is as defined in \eqref{eq: renormalizedV}, that $w$ admits $h\mapsto \rho \tilde a h^{\frac{1+s}{2}}$ as a modulus of growth, as the same can be said for $\tilde w$ and $h\mapsto \tilde a h^{\frac{1+s}{2}}$. Moreover, by a similar and coarser calculation as in the proof of Lemma \ref{lem: theApproximateSolutionMapConservesTheL2Norm}, we have $\norm{w}{L^2(\R)}\leq (1+\e^2\k^2)\norm{u}{L^2(\R)}$ where $\k=\norm{K}{L^1(\R)}$, and so $\rho\leq 1+\e^2\k^2$. 
Thus
\begin{align*}
    \rho\tilde{a}\leq (1+\e^2\kappa^2)\tilde a = a-\e f(a)+\e^2[C_A + \k^2\tilde a]\leq a-\e f(a) + \e^2\tilde{C}_A,
\end{align*}
where $\tilde{C}_A\coloneqq [C_A+\k^2(A+\e_A^2C_A)]$, and so this calculation shows that the proposition also holds for $w$ after choosing a larger constant $C_A$. 
\end{proof}

Together with a few results from Section \ref{sec: wellPosedness}, the previous proposition equips us with all we need to construct moduli of growth for entropy solutions of \eqref{eq: theMainEquation}. Roughly speaking, we can for small $\e>0$ iterate Proposition \ref{prop: newHolderCoefficientAfterBothSolutionMapsHaveBeenAdded} repeatedly to construct a modulus of growth for an approximate entropy solution \eqref{eq: definitionOfFamilyOfApproximateSolutions}, and further letting $\e\to0$ this construction carries over to the entropy solution itself. To formalize, we shall introduce some notation and assume from here on that a pair of constants $\e_A, C_A$, as described by Proposition \ref{prop: newHolderCoefficientAfterBothSolutionMapsHaveBeenAdded}, has been chosen for each $A>\underline{a}$. Define the function 
\begin{align}\label{eq: simplerNotationOfFatIterationFunction}
   g_{A}^\e(a)\coloneqq a- \e f(a) + \e^2C_A,
\end{align}
which is parameterized over $A>\underline{a}$ and $\e\in (0,\e_A]$ and where
\begin{align}\label{eq: simplerNotationOfIteratingFunction}
    f(a)=&\, \gamma a^{\frac{2-s}{2+s}}\Big(a^{\frac{3+2s}{2+s}}-\underline{a}^{\frac{3+2s}{2+s}}\Big), &
    \gamma=&\, \frac{1+s }{2^{\frac{2}{1+s}}c_s^{\frac{1-s}{1+s}}\mu^{\frac{1-s}{2+s}}}.
\end{align}
The function $f$ in \eqref{eq: simplerNotationOfIteratingFunction} is indeed the same as in \eqref{eq: definitionOfFunctionF}, and so $g_A^\e(a)$ is the `new Hölder coefficient' that Proposition \ref{prop: newHolderCoefficientAfterBothSolutionMapsHaveBeenAdded} provides. In the coming proposition, we carry out the above sketched argument consisting in part of repeated iterations of Proposition \ref{prop: newHolderCoefficientAfterBothSolutionMapsHaveBeenAdded}, and consequently, we will encounter repeated compositions of $g_A^\e$. We point out two relevant facts about $g_A^\e$. First off, for any $A>\underline{a}$ and sufficiently small $\e>0$, the function $g_A^\e$ maps $[\underline{a},A]$ to itself. To see this, note from \eqref{eq: simplerNotationOfFatIterationFunction} that $(g_A^\e)'$ is strictly positive on $[\underline{a},A]$ for small $\e>0$. Moreover, we have
\begin{align*}
    g_A^\e(\underline{a})=&\,\underline{a}, &  g_A^\e(A)=&\,A - \e f(A) + \e^2 C_A, 
\end{align*}
and since $f(A)>0$, it is clear that $\e>0$ can be made sufficiently small such that
\begin{align}\label{eq: gMapsIntervallToItself}
    \underline{a}=g_A^\e(\underline{a})\leq g_A^\e(a)\leq g_A^\e(A)\leq A,
\end{align}
for all $a\in [\underline{a},A]$. Our second fact, rigorously justified in the coming proposition, is that  repeated compositions of $g_A^\e$ applied to the starting value $a=A$ will, as $\e\to0$, result in a smooth function $a_A\colon[0,\infty)\to(\underline{a},A]$, implicitly defined by
\begin{align}\label{eq: implicitDefintionOfa_A(t)}
    t=\int_{a_A(t)}^A \frac{\mathrm{d}a}{f(a)}.
\end{align}
That \eqref{eq: implicitDefintionOfa_A(t)} yields a unique value $a_A(t)\in (\underline{a},A]$ for each $t\in[0,\infty)$ follows as the positive integrand has a non-integrable singularity at $a=\underline{a}$.
Alternatively, the function $a_A$ can be viewed as the solution of the differential equation
\begin{align}\label{eq: preciseReducedEquationWithoutH}
    \begin{cases}
           a'(t)= - f(a(t)),\quad t>0,\\
           a(0)=A.
    \end{cases}
\end{align}
For the next proposition, we shall exploit the two constants
\begin{align}\label{eq: twoUsefulConstants}
    M_{A} =&\, \max_{a\in[\underline{a},A]}|f'(a)|, & \tilde{M}_{A} =&\, \max_{a\in[\underline{a},A]}|f(a)f'(a)|,
\end{align}
both well defined as $f$ is smooth on $\R^+$. Note that the latter serves as a bound on $(a_A)''=f(a_A)f'(a_A)$, and so by Taylor expansion, we infer
\begin{align}\label{eq: approximatingTheSmootha_AWithDiscreteSteps}
    |a_A(t+\e) - a_A(t) + \e f(a_A(t))|\leq &\, \frac{\e^2}{2}\tilde{M}_A,
\end{align}
for all $t\geq 0$ and $\e\geq0$.

\begin{proposition}\label{prop: preciseModulusOfGrowthForEntropySolution}
Let $u$ be an entropy solution of \eqref{eq: theMainEquation}, whose initial data $u_0\in BV(\R)$ satisfies $\norm{u_0}{L^2(\R)}\leq \mu$ and admits a modulus of growth $h\mapsto A h^{\frac{1+s}{2}}$ for some $A>\underline{a}$. Then for all $t>0$, the function $x\mapsto u(t,x)$ admits the modulus of growth
\begin{align*}
    h\mapsto a_A(t)h^{\frac{1+s}{2}},
\end{align*}
with $a_A$ given by \eqref{eq: implicitDefintionOfa_A(t)}.
\end{proposition}
\begin{proof}
Consider $t>0$ fixed, and assume without loss of generality that $\norm{u_0}{L^2(\R)}<\mu$; if the proposition holds in this case, it necessarily also holds in the case $\norm{u_0}{L^2(\R)}\leq \mu$ as the implicit $\mu$-dependence of $a_A(t)$ is a continuous one. Pick a large $n\in\N$, set $\e=\tfrac{t}{n}$ and consider the family of functions $u_n^k\in BV(\R)$ defined inductively by
\begin{align*}
    \begin{cases}
    u_n^0 =S^B_\e(u_0),\\
    u_n^k= S^B_\e\circ S^K_\e(u^{k-1}_n),\quad k=1,2,\dots,n,
    \end{cases}
\end{align*}
As $u_0$ admits $h\mapsto Ah^{\frac{1+s}{2}}$ as a modulus of growth, so does $u_n^0$ by Lemma \ref{lem: generalEvolutionOfOleinikEstimateUnderBurgersSolutionMap}. Observe also that each $u^k_n\in BV(\R)$ as follows by induction and the properties of $S^B_\e$ and $S^K_\e$ listed at the very beginning in the proof of Proposition \ref{prop: propertiesOfTheApproximateSolutionMap}. Moreover, by similar reasoning as in the proof of Lemma \ref{lem: theApproximateSolutionMapConservesTheL2Norm}, we have
\begin{align*}
    \norm{u^k_n}{L^2(\R)}\leq e^{\frac{k}{2}\e^2\k^2}\norm{u_0}{L^2(\R)}\leq e^{\frac{t}{2n}\k^2}\norm{u_0}{L^2(\R)},\quad k=0,1,\dots,n,
\end{align*}
where $\k=\norm{K}{L^1(\R)}$. Since we have a strict inequality $\norm{u_0}{L^2(\R)}<\mu$, we can assume $n$ large enough such that $\norm{u^k_n}{L^2(\R)}\leq \mu$ for every $k$. We define further the coefficients $a_n^k$ inductively by
\begin{align*}
    \begin{cases}
    a_n^0 =A,\\
    a_n^k= g^\e_A(a_n^{k-1}),\quad k=1,2,\dots,n,
    \end{cases}
\end{align*}
where $g_A^\e$ is given by \eqref{eq: simplerNotationOfFatIterationFunction}. We assume $n$ large enough such that $\e=\frac{t}{n}$ is both less than $\e_A>0$ and small enough such that $g_A^\e$ maps $[\underline{a},A]$ to itself (see the discussion leading up to \eqref{eq: gMapsIntervallToItself}). In particular, each $a^k_n$ is in $[\underline{a},A]$. We may now apply Proposition \ref{prop: newHolderCoefficientAfterBothSolutionMapsHaveBeenAdded} inductively to each pair $(u_n^k,a_n^k)$, starting with $(u_n^0,a_n^0)$. As $u_n^0$ admits $h\mapsto a_n^0 h^{\frac{1+s}{2}}$ as a modulus of growth,
Proposition \ref{prop: newHolderCoefficientAfterBothSolutionMapsHaveBeenAdded} infers the same relationship for the pair $(u^1_n,a_n^1)$, and by repeating the argument, the same can be said for all pairs $(u_n^k,a_n^k)$. Most importantly, $u_n^n$ admits $h\mapsto a_n^n h^{\frac{1+s}{2}}$ as a modulus of growth. The proposition will now follow if we can, as $n\to\infty$, establish the limits
\begin{align}\label{eq: discreteModulusOfGrowthIsCloseToContinuous}
    a_n^n\to&\, a_A(t),\\\label{eq: discreteEntropySolutionIsCloseToContinuous}
     u_n^n\to&\, u(t),
\end{align}
where $u(t)=u(t,\cdot)$ and the latter limit is taken in $L^1_{\loc}(\R)$. Indeed, in this scenario we can let $\varphi$ denote any non-negative smooth function of compact support that satisfies $\int_\R \varphi \mathrm{d}x=1$ so to calculate for $h>0$
\begin{equation}\label{eq: localConvergenceImpliesModulusOfGrowthIsPreserved}
    \begin{split}
        \esssup_{x\in R}\Big[u(t,x+h)-u(t,x)\Big]=&\,\sup_{\varphi}\cop{u(t,\cdot+h)-u(t,\cdot),\varphi}\\
    =&\,\sup_{\varphi}\lim_{n\to\infty}\cop{u_n^n(\cdot+h)-u_n^n,\varphi}\\
    \leq &\,\sup_{\varphi}\lim_{n\to\infty}a_n^n h^{\frac{1+s}{2}}\\
    =&\, a_A(t)h^{\frac{1+s}{2}}.
    \end{split}
\end{equation}
We first prove \eqref{eq: discreteModulusOfGrowthIsCloseToContinuous}. Using the explicit form \eqref{eq: simplerNotationOfFatIterationFunction} of $g_A^\e$ with $\e=\frac{t}{n}$, the constants \eqref{eq: twoUsefulConstants} and the inequality \eqref{eq: approximatingTheSmootha_AWithDiscreteSteps} we can calculate for $k\geq1$,
\begin{equation}\label{eq: generalIterativeErrorOnSmoothAndDiscretea_A}
    \begin{split}
          &\,\Big|a_n^k- a_A\Big(\tfrac{kt}{n}\Big)\Big|\\
          = &\, \Big|g_A^\e\Big(a_n^{k-1}\Big) - a_A\Big(\tfrac{(k-1)t}{n}+\tfrac{t}{n}\Big)\Big|\\
          \leq&\,\Big|a_n^{k-1}-a_A\Big(\tfrac{(k-1)t}{n}\Big)\Big| + (\tfrac{t}{n})\Big|f\Big(a_n^{k-1}\Big)-f\Big(a_A\Big(\tfrac{(k-1)t}{n}\Big)\Big)\Big| + (\tfrac{t}{n})^2\Big(C_A + \tfrac{1}{2}\tilde{M}_A\Big)\\
    \leq &\, \Big[1+(\tfrac{t}{n})M_A\Big]\Big|a_n^{k-1}- a_A\Big(\tfrac{(k-1)t}{n}\Big)\Big| +  (\tfrac{t}{n})^2D_A,
    \end{split}
\end{equation}
with $D_A\coloneqq C_A + \frac{1}{2}\tilde{M}_A$. By repeated use of \eqref{eq: generalIterativeErrorOnSmoothAndDiscretea_A}, and the fact that $a_n^0=a_A(0)=A$, we conclude
\begin{align*}
 |a_n^n-a_A(t)|
 \leq  (\tfrac{t}{n})^2D_A\sum_{k=0}^{n-1}\Big[1+(\tfrac{t}{n})M_A\Big]^k
 \leq \tfrac{1}{n}\Big[t^2D_A e^{tM_A}\Big],
\end{align*}
and thus \eqref{eq: discreteModulusOfGrowthIsCloseToContinuous} is established. To prove \eqref{eq: discreteEntropySolutionIsCloseToContinuous}, we recall definition \eqref{eq: definitionOfApproximationOfSolutionMap} of the approximate solution map $S_{\e,t}$ and observe the relation
\begin{align}\label{eq: relationBetweenUnnAndUEpsilon}
    u_n^n=S^B_\e\circ S_{\e,t}(u_0)\eqqcolon S^B_\e(u^\e(t)),
\end{align}
where the definition of $u^\e$ coincides with \eqref{eq: definitionOfFamilyOfApproximateSolutions}, although we now work with a particular $u_0$ and $\e=\tfrac{t}{n}$. As Proposition \ref{prop: solutionWithBVDataIsLimitOfApproximateSolutions} ensures that $\lim_{\e\to 0} u^\e(t)=u(t)$ in $L^{1}_{\loc}(\R)$, the same limit then carries over to $u_n^n$ (as $n\to\infty$) by \eqref{eq: relationBetweenUnnAndUEpsilon} and the time continuity of the map $S^B_\e$ \eqref{eq: propertyFourOfBurgerMap} together with the $TV$ bound of $u^\e$ provided by Proposition \ref{prop: propertiesOfTheApproximateSolutionMap}.
With the two limits \eqref{eq: discreteModulusOfGrowthIsCloseToContinuous} and \eqref{eq: discreteEntropySolutionIsCloseToContinuous} established, the proof is complete.
\end{proof}

We may now prove Theorem \ref{thm: oneSidedHolderRegularity}.

\begin{proof}[Proof of Theorem \ref{thm: oneSidedHolderRegularity}]
We prove the theorem first for $u_0\in C_c^\infty(\R)$, and without loss of generality we assume $u_0\neq 0$. As the positive constant $\mu$ (introduced at the beginning of the subsection) was arbitrary, we may assume $\mu= \norm{u_0}{L^2(\R)}$. As $u_0\in C^\infty_c(\R)$, we infer from  Proposition \ref{prop: preciseModulusOfGrowthForEntropySolution} the existence of a sufficiently large $A$ such that $u$, the entropy solution of \eqref{eq: theMainEquation} corresponding to $u_0$, admits $h\mapsto a_A(t)h^{\frac{1+s}{2}}$ as a modulus of growth for all $t>0$. 

Observe that we have the following elementary inequality if $a-\underline{a}\geq 0$
\begin{align*}
   \Big(a-\underline{a}\Big)^{\frac{5+s}{2+s}} =    \Big(a-\underline{a}\Big)^{\frac{2-s}{2+s}}    \Big(a-\underline{a}\Big)^{\frac{3+2s}{2+s}} \leq  a^{\frac{2-s}{2+s}}\Big(a^{\frac{3+2s}{2+s}}-\underline{a}^{\frac{3+2s}{2+s}}\Big),
\end{align*}
where we for the second factor used that $x\mapsto x^p$ is super-additive when $x\geq 0$ and $p\geq 1$ (giving the desired conclusion for $x=a-\underline{a}$ and $p=\frac{3+2s}{2+s}$). Using this in \eqref{eq: implicitDefintionOfa_A(t)} gives
\begin{align}\label{eq: explicitExpressionFora_A}
   t\leq &\,\int_{a_A(t)}^A\frac{\dd a}{\gamma\Big(a-\underline{a}\Big)^{\frac{5+s}{2+s}}}
    =\frac{2+s}{3\gamma\Big(a_A(t)-\underline{a}\Big)^{\frac{3}{2+s}}}-\frac{2+s}{3\gamma\Big(A-\underline{a}\Big)^{\frac{3}{2+s}}}.
\end{align}
Removing the negative term on the right hand side and then rewriting \eqref{eq: explicitExpressionFora_A}, further gives 
\begin{align*}
a_A(t)\leq \underline{a} + \bigg(\frac{2+s}{3\gamma}\bigg)^{\frac{2+s}{3}}\frac{1}{t^{\frac{2+s}{3}}}\eqqcolon a(t).
\end{align*}
In particular, $u$ must also admit $h\mapsto a(t)h^{\frac{1+s}{2}}$ as a modulus of growth. By Lemma \ref{lem: relationBetweenQuantities}, we see that $a(t)$ may equivalently be written
\begin{align}\label{eq: theVerySimpleHolderCoefficient}
   a(t) = C_1(s) \kappa_s^{\frac{2+s}{3+2s}}\mu^{\frac{1+s}{3+2s}} + C_2(s)\frac{\mu^{\frac{1-s}{3}}}{t^{\frac{2+s}{3}}},
\end{align}
where $C_1(s)$ and $C_2(s)$ are given by \eqref{eq: C1AndC2}. This expression is exactly \eqref{eq: holderCoefficient} save for the fact that we have required $\k_s$ (introduced at the beginning of the subsection) to be greater or equal to $|K|_{TV^s}$ \textit{and} positive. Thus, we may not directly set $\kappa_s=|K|_{TV^s}$ if $K=0$. However, by $a(t)$'s continuous dependence on $\kappa_s$, it is clear that no problem may occur. Thus, Theorem \ref{thm: oneSidedHolderRegularity} follows for $C^\infty_c$-initial data.

Next, consider $u_0\in L^2\cap L^\infty(\R)$ and let $u$ denote the corresponding entropy solution of \eqref{eq: theMainEquation}. Pick a sequence of entropy solutions $(u_k)_{k\in\N}$ whose initial data $(u_{0,k})_{k\in\N}\subset C_c^\infty(\R)$ satisfies
\begin{align*}
    \norm{u_{0,k}}{L^2(\R)}\leq&\, \norm{u_0}{L^2(\R)}, & \norm{u_{0,k}}{L^\infty(\R)}\leq&\, \norm{u_0}{L^\infty(\R)},
\end{align*}
and yields the limit $\lim_{k\to\infty} u_{0,k}=u_0$ in $L^1_{\loc}(\R)$. By Proposition \ref{prop: uniquenessAndL1Contraction}, we then also get $\lim_{k\to\infty}u_k(t)= u(t)$ in $L^1_{\loc}(\R)$. Now Theorem \ref{thm: oneSidedHolderRegularity} carries over to $u$ by a calculation similar to \eqref{eq: localConvergenceImpliesModulusOfGrowthIsPreserved}, and so the theorem has been proved for $L^2\cap L^\infty$-initial data.

Finally, that this result can be extended to all weak solutions provided by Corollary \ref{cor: extensionToPureL2Data} follows by a density argument as above (using the continuity of the solution map $S$ of Corollary \ref{cor: extensionToPureL2Data} instead of the weighted $L^1$-contraction of Proposition \ref{prop: uniquenessAndL1Contraction}).
\end{proof}

\appendix

\section{Auxiliary results}
In the coming lemma we work with the concept of a modulus of growth as defined by Def. \ref{def: modulusOfGrowth}.
\begin{lemma}\label{lem: modulusOfGrowthImpliesLeftAndRightLimits}
Let $f\in L^1_{\loc}(\R)$ admit a modulus of growth $\omega$ that satisfies $\omega(0+)=0$. Then $f$ admits essential left and right limits at each point $x\in\R$. In particular, there are functions $f^-$ and $f^+$, respectively left- and right-continuous, that coincides a.e. with $f$.
\end{lemma}
\begin{proof}
For any $x\in\R$ the existence of an essential left limit $f(x-)$ of $f$ at $x$, follows from the calculation
\begin{align*}
    &\esslimsup_{\substack{y<0\\ y\to 0}} f(x+y) - \essliminf_{\substack{y<0\\ y\to 0}} f(x+y)\\
    =&\esslimsup_{\substack{y_2< y_1< 0\\ y_2,y_1\to0}}\Big[f(x+y_1)-f(x+y_2)\Big]\\
    \leq& \limsup_{\substack{y_2< y_1< 0\\ y_2,y_1\to0}}\omega(y_1-y_2)=0.
\end{align*}
By the Lebesgue differentiation theorem, the function $f^{-}(x)\coloneqq f(x-)$ can only differ from $f$ on a null set, and moreover, must be left continuous as the above calculation could be repeated for $f^-$ with essential limits replaced by limits. A similar argument yields the existence of an essential right limit $f(x+)$ of $f$ at each $x\in\R$ and further that $f^+(x)\coloneqq f(x+)$
is a right-continuous function agreeing a.e. with $f$.
\end{proof}

The next lemma deals with quantities appearing throughout the paper and the relations between them. For convenience, we here list the definition of each relevant quantity; some of them given for the first time. The quantities $c_s$ and $\gamma$ were in \eqref{eq: definitionOfCs} and \eqref{eq: simplerNotationOfIteratingFunction} defined to be
\begin{align*}
        c_s =&\, \bigg[\frac{(2+s)(3+s)}{2(1+s)^2}\bigg]^{\frac{1+s}{2(2+s)}}, & \gamma = &\, \frac{1+s }{2^{\frac{2}{1+s}}c_s^{\frac{1-s}{1+s}}\mu^{\frac{1-s}{2+s}}}.
\end{align*}
We also introduce the expressions $C_1(s)$ and $C_2(s)$ by 
\begin{align}\label{eq: C1AndC2}
    C_1(s)\coloneqq&\,\frac{2^{\frac{3+s}{6+4s}}[(2+s)(3+s)]^{\frac{1+s}{6+4s}}}{1+s}, & C_2(s)\coloneqq&\,\frac{2^{\frac{4+2s}{3+3s}}(2+s)^{\frac{5+s}{6}}(3+s)^{\frac{1-s}{6}}}{2^{\frac{1-s}{6}}3^{\frac{2+s}{3}}(1+s)}.
\end{align}
In the coming lemma, we will also see the quantities $\mu$ and $\k_s$; these are simply placeholders for the expressions $\norm{u_0}{L^2(\R)}$ and $|K|_{TV^s}$ respectively and will not affect the algebra in any non-trivial way.

\begin{lemma}\label{lem: relationBetweenQuantities}
With $c_s, \gamma, C_1(s), C_2(s), \mu$ and $\k_s$ as they appear above, we have the relations
\begin{align}\label{eq: relationOfLimitA}
   \underline{a}\coloneqq \bigg(\frac{2c_s\kappa_s}{1+s}\bigg)^{\frac{2+s}{3+2s}}\mu^{\frac{1+s}{3+2s}}=&\,C_1(s)\kappa_s^{\frac{2+s}{3+2s}}\mu^{\frac{1+s}{3+2s}},\\
   \label{eq: relationOfC2}
   \bigg(\frac{2+s}{3\gamma}\bigg)^{\frac{2+s}{3}} = C_2&(s)\mu^{\frac{1-s}{3}}.
\end{align}
\end{lemma}
\begin{proof}
We start with \eqref{eq: relationOfLimitA}: inserting for $c_s$ on the left-hand side of \eqref{eq: relationOfLimitA} we get
\begin{align*}
    &\,\bigg(\frac{2}{1+s}\bigg)^{\frac{2+s}{3+2s}}\bigg(\frac{(2+s)(3+s)}{2(1+s)^2}\bigg)^{\frac{1+s}{2(3+2s)}}\kappa_s^{\frac{2+s}{3+2s}}\mu^{\frac{1+s}{3+2s}}\\
    =&\,\underbrace{\Bigg[\frac{2^{\frac{3+s}{6+4s}}[(2+s)(3+s)]^{\frac{1+s}{6+4s}}}{1+s}\Bigg]}_{C_1(s)}\kappa_s^{\frac{2+s}{3+2s}}\mu^{\frac{1+s}{3+2s}},
\end{align*}
and so \eqref{eq: relationOfLimitA} is established. Second, we prove \eqref{eq: relationOfC2}:  if we on the left-hand side of \eqref{eq: relationOfC2} insert for $\gamma$ we get 
\begin{align*}
\Bigg(\frac{(2+s)^{\frac{2+s}{3}}2^{\frac{2(2+s)}{3(1+s)}}}{3^{\frac{2+s}{3}}(1+s)^{\frac{2+s}{3}}}\Bigg)c_s^{\frac{(1-s)(2+s)}{3(1+s)}}\mu^{\frac{1-s}{3}}.
\end{align*}
Further inserting for $c_s$ we obtain
\begin{align*}
&\,\Bigg(\frac{(2+s)^{\frac{2+s}{3}}2^{\frac{2(2+s)}{3(1+s)}}}{3^{\frac{2+s}{3}}(1+s)^{\frac{2+s}{3}}}\Bigg)\bigg(\frac{(2+s)(3+s)}{2(1+s)^2}\bigg)^{\frac{1-s}{6}}\mu^{\frac{1-s}{3}}\\
=&\,\underbrace{\Bigg[\frac{2^{\frac{4+2s}{3+3s}}(2+s)^{\frac{5+s}{6}}(3+s)^{\frac{1-s}{6}}}{2^{\frac{1-s}{6}}3^{\frac{2+s}{3}}(1+s)}\Bigg]}_{C_2(s)}\mu^{\frac{1-s}{3}},
\end{align*}
and so \eqref{eq: relationOfC2} is established. 
\end{proof}

\section{Proof of Corollary \ref{cor: decayingHeightBoundForEntropySolutions} and Corollary \ref{cor: maximalLifespanForClassicalSolutions}}\label{sec: proofsOfCorollaries}
We prove Corollary \ref{cor: decayingHeightBoundForEntropySolutions} which provides a decaying $L^\infty$ bound for entropy solutions of \eqref{eq: theMainEquation}.
\begin{proof}[Proof of Corollary \ref{cor: decayingHeightBoundForEntropySolutions}]
By the $s=0$ case of Theorem \ref{thm: oneSidedHolderRegularity} we know that $u(t)$ admits the modulus of growth (Def. \ref{def: modulusOfGrowth}) $h\mapsto a(t)h^{\frac{1}{2}}$, where $a(t)$ is given by 
\begin{align*}
    a(t)=2^{\frac{4}{3}} 3^{\frac{1}{6}}\norm{K}{L^1(\R)}^{\frac{2}{3}}\norm{u_0}{L^2(\R)}^{\frac{1}{3}} +  \frac{4\norm{u_0}{L^2(\R)}^{\frac{1}{3}}}{3^{\frac{1}{2}}t^{\frac{2}{3}}}.
\end{align*}
This expression is precisely what is provided by \eqref{eq: holderCoefficient} when using $C_1(0)=2^{\frac{2}{3}}3^{\frac{1}{6}}$, $C_2(0)=4/3^{\frac{1}{2}}$ and $|K|_{TV^0}=2\norm{K}{L^1}$. Setting $\mu=\norm{u_0}{L^2(\R)}$ in Lemma \ref{lem: explicitBoundOnSupFromExplicitOmegaFunction} and using $\norm{u(t)}{L^2(\R)}\leq \norm{u_0}{L^2(\R)}$ we infer from said lemma -- more specifically \eqref{eq: explicitBoundOnSupFromExplicitOmegaFunction} -- that
\begin{align*}
     \norm{u(t)}{L^\infty(\R)}\leq&\, 2^{\frac{1}{4}}3^{\frac{1}{4}} \norm{u_0}{L^2(\R)}^{\frac{1}{2}}a(t)^{\frac{1}{2}},
\end{align*}
for all $t>0$, where we used that $c_s=3^{\frac{1}{4}}$ when $s=0$. Using the sub-additivity of $y\mapsto |y|^{\frac{1}{2}}$ we infer that
\begin{align*}
    a(t)^{\frac{1}{2}} \leq 2^{\frac{2}{3}} 3^{\frac{1}{12}}\norm{K}{L^1(\R)}^{\frac{1}{3}}\norm{u_0}{L^2(\R)}^{\frac{1}{6}} + \frac{2\norm{u_0}{L^2(\R)}^{\frac{1}{6}}}{3^{\frac{1}{4}}t^{\frac{1}{3}}},
\end{align*}
and so inserting this in the above inequality we get
\begin{align*}
    \norm{u(t)}{L^\infty(\R)}\leq 2^{\frac{11}{12}}3^{\frac{1}{3}}\norm{K}{L^1(\R)}^{\frac{1}{3}}\norm{u_0}{L^2(\R)}^{\frac{2}{3}} + \frac{2^{\frac{5}{4}}\norm{u_0}{L^2(\R)}^{\frac{2}{3}}}{t^{\frac{1}{3}}},
\end{align*}
for all $t>0$. 
\end{proof}

Next, we prove Corollary \ref{cor: maximalLifespanForClassicalSolutions} which established a maximal lifespan for classical solutions of \eqref{eq: theMainEquation} with $L^2\cap L^\infty$ data.

\begin{proof}[Proof of Corollary \ref{cor: maximalLifespanForClassicalSolutions}]
Consider $s\in[0,1]$ fixed for now, and assume $|K|_{TV^s}<\infty$.
As (bounded) classical solutions are entropy solutions, we may associate $u\in L^\infty\cap C^1((0,T)\times \R)$ with the global entropy solution admitting $u_0$ as initial data, provided by Theorem \ref{thm: wellPosedness}; the discussion following the proof of Proposition \ref{prop: uniquenessAndL1Contraction} justifies this viewpoint. Referring to this solution also as $u$, we have by \eqref{eq: boundsForWellPosednessThm} that $x\mapsto u(T,x)$ is a well defined element of $L^2\cap L^\infty(\R)$ approximated in $L^2$ sense by $u(t)$ as $t\nearrow T$. Setting $v(t,x)\coloneqq u(T-t,-x)$, we see through pointwise evaluation that $v$ also is a classical solution of \eqref{eq: theMainEquation} (and thus an entropy solution) on $(0,T)\times \R$ with initial data $v_0(x)\coloneqq u(T,-x)$. From \eqref{eq: boundsForWellPosednessThm} we then infer $\norm{v_0}{L^2(\R)}=\norm{u_0}{L^2(\R)}$ since
\begin{align*}
    \norm{v_0}{L^2(\R)}= \norm{u(T)}{L^2(\R)}\leq \norm{u_0}{L^2(\R)}=\norm{v(T)}{L^2(\R)}\leq \norm{v_0}{L^2(\R)}.
\end{align*}
Using the identity $u_0(x)=v(T,-x)$ for a.e. $x\in\R$ and applying Theorem \ref{thm: oneSidedHolderRegularity} to $v$ we further find for all $h>0$ and a.e. $x\in \R$ that
\begin{align}\label{eq: backwardsHolderEstimateForV}
   u_0(x-h)-u_0(x) = v(T,-x+h)-v(T,-x)\leq a(T)h^{\frac{1+s}{2}},
\end{align}
where $a(T)$ is given by
\begin{align}\label{eq: simpleHolderConditionForProofOfMaximalLifespan}
    a(T)=C_1(s) |K|_{TV^s}^{\frac{2+s}{3+2s}}\norm{u_0}{L^2(\R)}^{\frac{1+s}{3+2s}} + C_2(s)\frac{\norm{u_0}{L^2(\R)}^{\frac{1-s}{3}}}{T^{\frac{2+s}{3}}}\eqqcolon \underline{a} + \frac{q}{T^{\frac{2+s}{3}}},
\end{align}
and where we have substituted $\norm{u_0}{L^2(\R)}$ for $\norm{v_0}{L^2(\R)}$ as the two quantities agree. Dividing each side of \eqref{eq: backwardsHolderEstimateForV} by $h^{\frac{1+s}{2}}$ and taking the essential supremum with respect to $x\in \R$ we get
\begin{align}\label{eq: initialDataMustHaveSomeHolderRegularityForClassicalSolution}
    [u_0]_s\coloneqq \esssup_{\substack{x\in\R\\ h>0}}\bigg[\frac{u_0(x-h)-u_0(x)}{h^{\frac{1+s}{2}}}\bigg]\leq \underline{a} + \frac{q}{T^{\frac{2+s}{3}}},
\end{align}
and if $[u_0]_s > \underline{a}$ then \eqref{eq: initialDataMustHaveSomeHolderRegularityForClassicalSolution} can be rewritten as
\begin{align}\label{eq: slightlyPreciseMaximalLifespan}
    T\leq \Bigg[\frac{q}{[u_0]_s-\underline{a}}\Bigg]^{\frac{3}{2+s}} = \Bigg(\frac{C_2(s)}{1-\tfrac{\underline{a}}{[u_0]_s}}\Bigg)^{\frac{3}{2+s}}\frac{\norm{u_0}{L^2(\R)}^{\frac{1-s}{2+s}}}{[u_0]_s^{\frac{3}{2+s}}}\eqqcolon F\Big(\tfrac{\underline{a}}{[u_0]_s}\Big)\frac{\norm{u_0}{L^2(\R)}^{\frac{1-s}{2+s}}}{[u_0]_s^{\frac{3}{2+s}}},
\end{align}
where the first equality replaced $q$ by its explicit expression as given by \eqref{eq: simpleHolderConditionForProofOfMaximalLifespan}. We now show that this gives for any $\rho\in(0,1)$ the following implication
\begin{align}\label{eq: chainOfImplications}
    [u_0]_s^{3+2s}>\bigg(\frac{C_1(s)}{\rho}\bigg)^{3+2s}|K|_{TV^s}^{2+s}\norm{u_0}{L^2(\R)}^{1+s}, \quad \implies \quad T\leq F(\rho)\frac{\norm{u_0}{L^2(\R)}^{\frac{1-s}{2+s}}}{[u_0]_s^{\frac{3}{2+s}}}.
\end{align}
Indeed, using the explicit expression \eqref{eq: simpleHolderConditionForProofOfMaximalLifespan} for $\underline{a}$ we see that the left-hand side of \eqref{eq: chainOfImplications} is equivalent to $[u_0]_s>\underline{a}/\rho$ which, as $\rho\in (0,1)$, implies that $[u_0]_s>\underline{a}$ and so \eqref{eq: slightlyPreciseMaximalLifespan} holds. By observing that $\rho\mapsto F(\rho)$ is strictly increasing on $(0,1)$ and that $\rho>\underline{a}/[u_0]_s$ we see that the right-hand side of \eqref{eq: chainOfImplications} then follows from \eqref{eq: slightlyPreciseMaximalLifespan}. With \eqref{eq: chainOfImplications} established,
the corollary follows: for any $\rho\in(0,1)$ we get such universal constants $c$ and $C$ by setting 
\begin{align}\label{eq: theConclusionOfCorollary}
   c=&\, \sup_{s\in[0,1]}\bigg(\frac{C_1(s)}{\rho}\bigg)^{3+2s},  &
    C=&\,\sup_{s\in[0,1]}F(\rho)=\sup_{s\in[0,1]} \Bigg(\frac{C_2(s)}{1-\rho}\Bigg)^{\frac{3}{2+s}}.
\end{align}
The free parameter $\rho$ allows us to shrink one of the two constants at the cost of enlarging the other; in particular, $c$ is at its smallest for $\rho\to 1$ while $C$ is at its smallest for $\rho\to 0$.
\end{proof}

\section*{Acknowledgement}
The first author acknowledges the support of The Research Council of Norway through the project INICE (301538).

\bibliographystyle{siam}
\bibliography{references}

\end{document}